\documentclass[11pt,amsmath,amsthm,amsfonts,amssymb,amscd]{article}
\usepackage{epsfig}
\usepackage[latin1]{inputenc}

\usepackage{amsthm}
\usepackage{amsmath}
\usepackage{graphicx}
\usepackage{srcltx}
\usepackage{psfrag}
\font \Bbbten=msbm10 \font \Bbbsev=msbm7 \font \Bbbfiv=msbm5
\newfam\Bbbfam \textfont\Bbbfam = \Bbbten
\scriptfont\Bbbfam=\Bbbsev \scriptscriptfont\Bbbfam=\Bbbfiv
\newcommand{\N}{\mbox{$I\!\!N$}}
\newcommand{\Z}{\mbox{$Z\!\!\!Z$}}

\newcommand{\R}{\mbox{$I\!\!R$}}

\newcommand{\dist}{{\rm dist}}

\newcommand{\cita}[7]{{\sc #1, }{\it #2, }{\small #3, { #4 } (#5), p.
#6-#7.}}
\newcommand{\cit}[5]{{\sc #1, }{\it #2, }{\small #3, {\bf #4 } #5.}}

\theoremstyle{plain}
\newtheorem{thm}{Theorem}[section]

\newtheorem{claim}[thm]{Claim}

\newtheorem{Df}{Definition}[section]
\newtheorem{Teo}{Theorem}[section]
\newtheorem{Lem}[Teo]{Lemma}
\newtheorem{Cor}[Teo]{Corollary}
\newtheorem{Prop}[Teo]{Proposition}
\newtheorem{Obs}[Teo]{Remark}

\title{Evolution of the population of {\it Microtus Epiroticus}: the Yoccoz-Birkeland model.}
\author{J. J. Nieto\footnote{J.J. Nieto is  partially supported by Ministerio de Educacion y Ciencia, Spain, and FEDER, Project MTM2010-15314.} \quad M. J. Pacifico\footnote{
  M. J. Pacifico is partially supported by CNPq, FAPERJ and PRONEX Dynamical Systems, Brazil.} \quad J. L. Vieitez\footnote{ J. L. Vieitez is partially supported by PEDECIBA and ANII CE 10065, Uruguay.}
  }

\date{\today}

\begin{document}
\maketitle

\begin{abstract}
 We study the discretized version of a dynamical system given by a model proposed by Yoccoz and Birkeland
to describe the evolution of the population of
 {\it Microtus Epiroticus}  on Svalbard Islands,
 see ${\rm http://zipcodezoo.com/Animals/M/Microtus\_\,epiroticus}$.
 We prove that this discretized version has an attractor $\Lambda$ with a hyperbolic 2-periodic point $p$ in it. For certain values of the parameters the system restricted to the attractor exhibits sensibility to initial conditions. Under certain assumptions that seems to be sustained by numerical simulations, the system is
topologically mixing (see definition \ref{topmix}) explaining some
of the high oscillations observed in Nature. Moreover, we estimate its order-2 Kolmogorov entropy obtaining a positive value. Finally we give numerical evidence that there is a homoclinic point associated with $p$.
\end{abstract}


\setcounter{tocdepth}{2}
\tableofcontents

\section{Introduction}
We study the evolution of the population of {\it Microtus
Epiroticus} (sibling vole) on Svalbard Islands in the Arctic Ocean,
using a model proposed by J. C. Yoccoz and  H. Birkeland, see \cite{Ar}.
It is known that there are no significant predation of these small
mammals but in spite of that, the population presents high
oscillations in its number albeit the lack of food is not a
determinant factor to the occurrence of these phenomena. This population exhibits
dramatic multi-annual fluctuations, by a factor greater than 20, \cite{YI}.

The Sibling Vole ({\it Microtus Epiroticus}) is a species of vole found through much
 of northern Europe. First discovered in 1960 in the Grumantbyen area,
 they were thought to be the Common Vole until a genetic
 analysis correctly identified them in 1990, \cite{FJASY}.

Since these rodents were introduced from Russia on Svalbard Isles between 1930
and 1960, \cite{YI},
the annual oscillations of their number may be explained, at least in part, by a non
total adaptation to the environment, and by the pronounced seasonal fluctuation in
climatic variability at Svalbard where temperatures of $-30$ degrees Celsius
are common, see \cite{YI, LBY}.

\begin{figure}[ht]
\begin{center}
\includegraphics[scale=0.30]{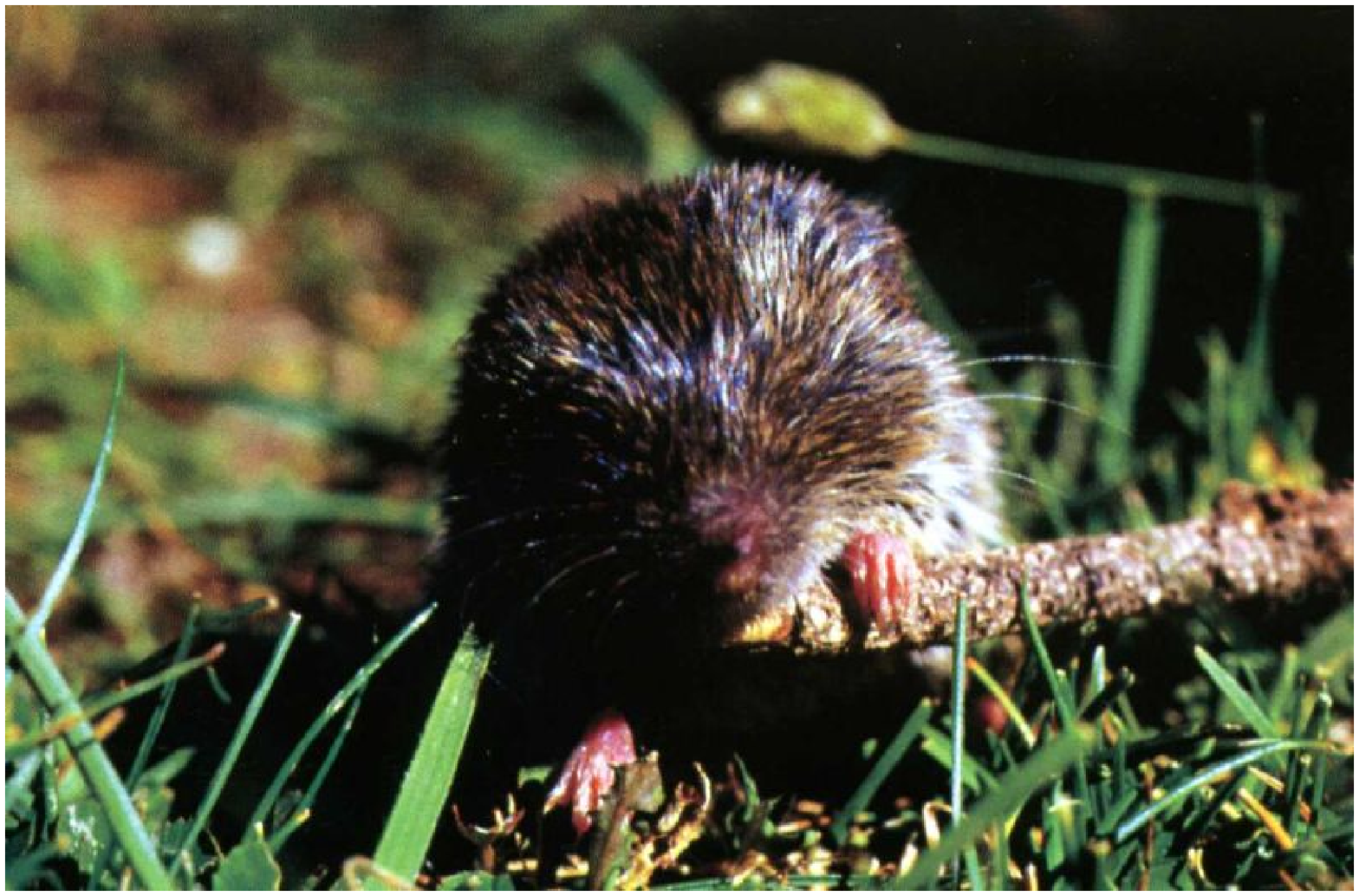}
\includegraphics[scale=0.32]{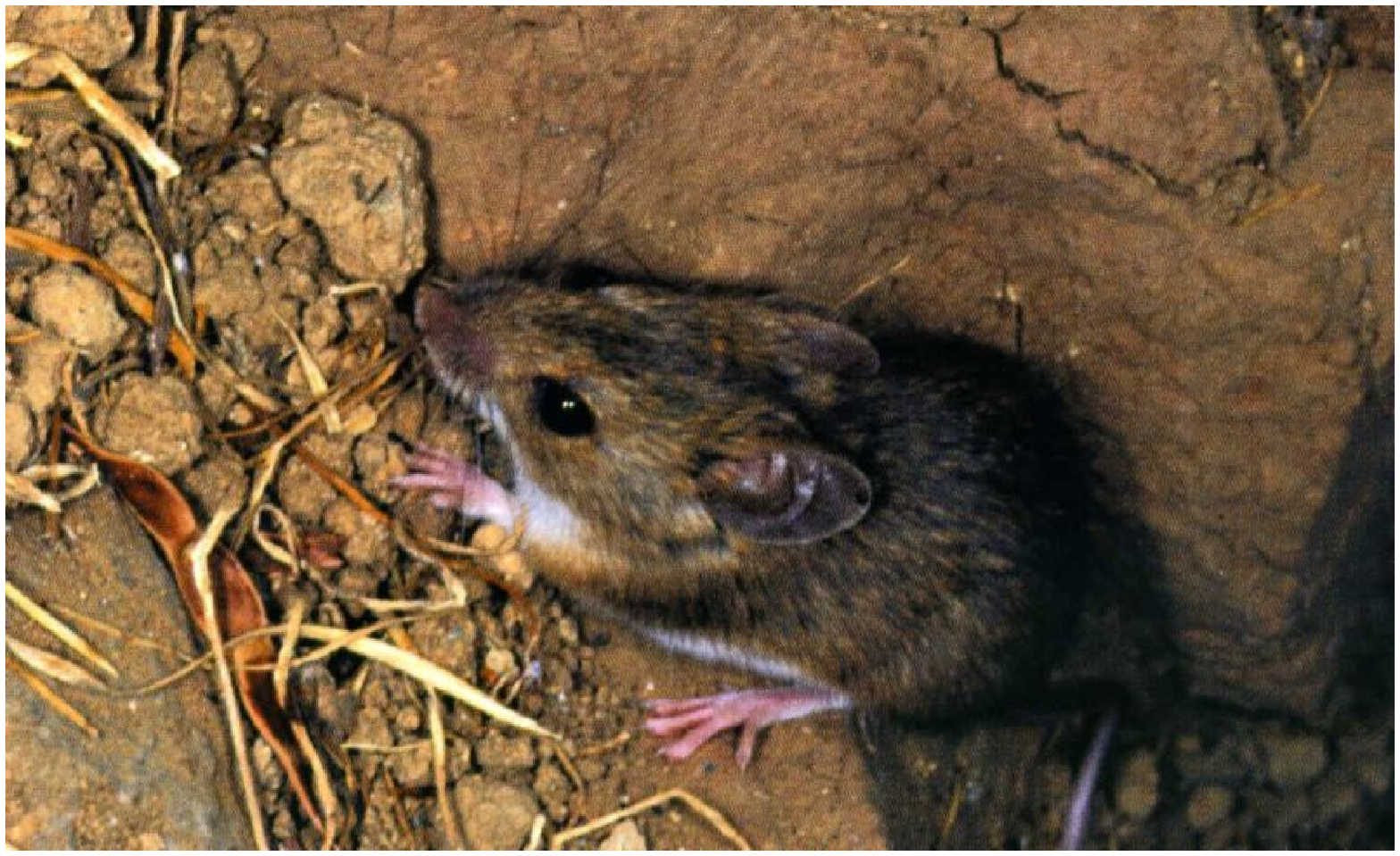}
\caption{Microtus Epiroticus.}\label{fig2}
\end{center}
\end{figure}

Let us first sketch the taxonomy classification of {\it Microtus Epiroticus}.
\begin{itemize}
\item
Domain:   Eukaryota
 \item
\quad Kingdom:  Animalia
\item
\quad$\quad $ Phylum:   Chordata
\item
\quad$\quad\quad$ Class:    Mammalia
\item
\quad$\quad\quad\quad$ Order:   Rodentia
\item
\quad$\quad\quad\quad\quad$ Family:     Muridae
\item
\quad$\quad\quad\quad\quad\quad$ Subfamily:     Arvicolinae
\item
\quad$\quad\quad\quad\quad\quad\quad$\quad Genus: \;\;  Microtus
\item
\quad$\quad\quad\quad\quad\quad\quad\quad$\quad\; Species:  \; Microtus Epiroticus
\end{itemize}

\vspace*{1cm}

Jean Christophe Yoccoz and H. Birkeland, see \cite{Ar}, have proposed the following
equation
\begin{equation}
\label{rato} N(t)=\int_{A_0}^{A_1}N(t-a)m(N(t-a))m_\rho(t-a)S(a)da\,
,\end{equation}
to model the evolution in time of the population of {\it Microtus Epiroticus}. In
the equation it is taken
 into account only the number $N(t)$ of fertile females at certain time $t$.
 Indeed, the inclusion in a model of the number of males is justified when there
are difficulties for a female to find a male (for instance if the
density of population is too small or if the ratio male-population :
female-population is far away from $1:1$) which is not the case for
these rodents. In fact as has been pointed out by R. A. Ims in \cite{Ims},
``spatial clumping of sexually receptive females induces space sharing among male voles''
which implies that it is not difficult for a female to find a male.
Moreover, the quantity of females is about the same as those of males
for these rodents, \cite{Ims2, YI}.

\medbreak

\noindent Let us describe the parameters of the model given by  equation (\ref{rato}):

\begin{enumerate} \label{donde}
\item $t$: is the time measured in years.
\item $N(t)$: is the population of active females at time $t$.
\item $A_0$: is the maturation age,
%more precisely when a female has its first litter of pups.
\item $A_1$: is the maximal age expected for {\it Microtus
Epiroticus}.
\item $m(N)$: annual individual reproduction rate for a population
of $N$ individuals
\item $m_\rho(t)$: is the reproduction probability at time $t$ of
the year.
\item $S(a)$: probability to survive up to $a$ years.
\end{enumerate}
The model take into account the following facts:

\begin{enumerate}
\item[(a)]  The age when the females of {\it Microtus Epiroticus} have their first
offspring
 is about 50 days, i.e., $A_0\approx 0.14$ years (see \cite{YIS}).
\item[(b)]
The maximal age of survival is about 2 years, i.e.,  $A_1=
2$ (see \cite{YI}).
\item[(c)]
The seasonal factor $m_\rho(t)$, that is, the reproduction probability at
time $t$ of the year, varies sharply from 0 in Winter to 1 from
Spring to Autumn. Thus, the definition of $m_\rho(t) $ we adopt is
$$m_\rho(t)=\left\{\begin{array}{c}
    0\quad \mbox{if }\; 0\leq t\!\!\mod (1)<\rho  \\
    1\quad\mbox{if }\; \rho\leq t\!\!\mod (1)<1
  \end{array}\right.
\, .$$
\item[(d)]
The annual individual reproduction rate $m(N)$ for a population of
$N$ individuals, is too high when $N(t)$ is small. Indeed $m(N)$ of the order of a constant
$m_0>30$ individuals is realistic due to the high
fertility of these
   rodents. The value of $m(N(t))$ decays sharply
when the population $N(t)$ increases.
 Following \cite{Ar}, for $m(N)$ we adopt
 \begin{equation} \label{emedeene} m(N)=\left\{\begin{array}{c}
     m_0\qquad \mbox{if }\; N\leq 1 \\
     m_0N^{-\gamma}\quad \mbox{if }\; N>1
   \end{array}\right. ;\quad \gamma
   > 1 \, . \end{equation}

  We will assume that $\gamma >1$ and
   for some calculations we take  $\gamma=8.25$. The reason for that is that there is numerical evidence, see \cite{Ar}, that for this value of the parameter we have chaotic behavior.

\item[(e)] Finally for the survival probability $S(a)$, again following \cite{Ar}, we consider
a linear function:
$$S(a)=1-\frac{a}{A_1},\quad \mbox{ if }\quad 0\leq a\leq A_1,\quad \mbox{and }\quad S(a)=0
\quad\mbox{elsewhere}\, . $$
\end{enumerate}
\begin{Obs} \label{nota}
Another choice of functions for $S(a)$, for instance $S(a)=\exp(-\kappa a)$,
with $\kappa
>0$, are also usual in the literature.
It would be interesting to test the model given by {\rm
(\ref{rato})} replacing the linear function at {\rm (e)} by a
exponential one.
\end{Obs}
Let us describe how the integral equation
 $$N(t)=\int_{A_0}^{A_1}N(t-a)m(N(t-a))m_\rho(t-a)S(a)da\,\quad \mbox{arises}\, .$$
For $N(t)$, the contribution of females of age in between
$[a,a+\Delta a]\subset  [A_0,A_1]$  is
$$\mbox{fem}(t-a)\times (\mbox{reprod.
rate}(t-a)\times (\mbox{ season factor}(t-a)\times (\mbox{prob.
survive})\times \ell ([a,a+\Delta a])$$
$$=N(t-a)\times m(N(t-a))\times m_\rho(t-a)\times S(a)\times \Delta
a\, ,$$
 where $\ell(J)$ is the length of the interval $J$ and fem$(t)$
 is the number of females at time $t$.
Here we assume that a female of age
near $A_1$ can reproduce and $\Delta a$ is small.
Taking a partition $\{a_0=t-A_1,a_1,\ldots, a_n=t-A_0\}$ of the
interval $[t-A_1,t-A_0]$ we find
$$N(t)\approx\sum_{j=0}^{n-1}N(t-a_j) m(N(t-a_j)) m_\rho(t-a_j)
S(a_j) \Delta a_j\, , \quad \mbox{where}\quad\Delta a_j=(a_{j+1}-a_j)$$
Letting $n\to\infty$ we get at the limit the integral equation
given by (\ref{rato}).

\medbreak

\subsection{The discrete model.}
 There is no special reason to prefer the
continuous model above to its discretization: most of the quantities
involved, as $N(t)$ and $m(N)$, are by nature of discrete type.
Moreover, from the experimental point of view, it is more natural to
split the year on days and even in groups of days since it
is very difficult to monitor $N(t)$ experimentally. Hence, we
assume that the year is split into $p$ equal parts.

Since the
expected value of survival is bounded by $A_1=2$ years we will study
 the evolution of $N(t)$ for discrete values of $t$,
modeling the period $[0,A_1]$ as a vector of $A_1p+1$ real entrances,
from $t=0$ at the initial time of the first year, to $t=2p$
corresponding to $A_1=2$ the final time of the second year.

In this case the probability of survival at age $\frac{j}{p}$ is
given by $$S(j)=1-\frac{j}{2p}\, ,\quad j=0,1,2\ldots ,2p\,.$$ where
$A_1p=2p$.
%In order to consider the fact that part of the population
%at instant $t$ is given by females born $A_1$ years before, we will
%modify slightly the function $S(j)$ to have $S(2p)>0$. For this we
%consider either
It is also convenient to consider
$S(j)=1-\frac{j}{2p+1}$. This takes into account the case where $S(2p)>0$,
i.e., when these animals can reproduce till the final of their lives.

\medbreak

Given an initial vector value $(N_0,N_1,N_2,\ldots, N_{A_1p-1},
N_{A_1p})\in\R^{A_1p+1}\,$, the evolution of $N(t)=N_t$, $t\in\N$, is
governed by

\begin{equation} \label{clave}
N_t=\sum_{h=A_0\,p}^{A_1\,p-1}N_{t-h} m(N_{t-h}) m_{\rho}(t-h) S(h)
\Delta h
\end{equation}
$$ =\frac{1}{p}\sum_{h=A_0\,p}^{2\,p-1}N_{t-h} m(N_{t-h}) m_{\rho}(t-h) S(h)\, .$$

Next we explain the choices in  equation (\ref{clave}).
\begin{enumerate}
\item
 We take $A_0p= \left[\frac{50\times
p}{365}\right]\approx 14$  which corresponds to the age at which the
females have their first litter of pups (about 50 days). Note that if $p=100$ then
$A_0=0.14$ corresponds to $51$ days.

\item We take  $\Delta
h=\frac{1}{p}$ years that corresponds to the length of the unit interval
in which we split the year. When $p=100$ this gives $\Delta
h=\frac{1}{100} \mbox{ years}=3.65$ days.

\end{enumerate}
Note that the value of $N$ at $t$ depends only on the values of $N$
in $[t - A_1; t - A_0]$. Thus, the knowledge of $N_t$ for $t \in
[-A_1 p,0]$ (two years of observation) enables us to predict $N_t$
for $t\in[0,A_0 p]$. When $p=100$ and $A_0=0.14$ this means that the
knowledge of $(N_0,N_1,N_2,\ldots, N_{200})$ enables us to compute
$N_{201},\ldots ,N_{214}$.
 Recursively we may compute
$N_j$ for all $j\geq 0$.

%\begin{Obs}
%In fact the value of $N_{A_1p}=N_{2p}$ is determined by equation {\rm (\ref{clave})} and the value of
%$(N_0,N_1,N_2,\ldots, N_{A_1p-1})$.
%\end{Obs}

\section{The dynamical system}
Equation (\ref{clave}) defines a discrete dynamical system in
$\R^{2p+1}$ as follows:
$$(N_0,N_1,\ldots ,N_{2p})\mapsto T(N_0,N_1,\ldots ,N_{2p})=
(N_{p},N_{p+1},\ldots ,N_{3p})\,,$$ where we have used that $A_1=2$ and $T:\R^{2p+1}\to\R^{2p+1}$
 is defined recursively by equation (\ref{clave}) for $t=2p+1,\ldots, 3p$.

In order to describe theoretical properties of a system given
by the discretized version (\ref{clave})
 of Yoccoz-Birkeland equation (\ref{rato}),
let us assume the following restrictions that weaken those given by
conditions (a)--(e) described before. Doing this allows to apply
the conclusions to different species respecting equation (\ref{clave})
and those restrictions. In particular, these conclusions
will apply to the original system modeling {\it Microtus Epiroticus}.
% but that are enough for our
%purposes.
\begin{enumerate}
\item $m(N)$ is a continuous function, $m:\R^+\to\R^+$,
\item there is $m_0\in\R^+$ such that
\begin{equation} \label{eq2} \left\{\begin{array}{l}
  m_0\geq m(N)\geq m_0/2\qquad\qquad\qquad\qquad\quad\;\, \mbox{ if }\quad N\leq 1\qquad\mbox{and} \\
  m_0N^{-\gamma}\geq m(N)\geq
  \min\{\frac{m_0}{2},m_0\cdot N^{-\gamma}\}\quad\;\!\mbox{ if }\quad N> 1\,\,,
\end{array}\right.
\end{equation}
\item $0\leq m_\rho(t)\leq 1$ (so that we now allow $0<m_\rho(t)<1$
for certain values of $t$),
\item There is $\epsilon \geq 0$ such that $m_\rho(t)=1$ for $t$ in an interval of length
$1-\rho-\epsilon>0$,
\item  $0<2A_0<A_1$ and $A_0+1<A_1$ (this means that in average each individual has at least
two opportunities to reproduce),
\item
Defining $c_0$ as
$c_0=\frac{1}{p}\left(\sum_{h=A_0p+(\rho+\epsilon)p}^{A_0p+p}S(h)\right)$
we require $c_0\,m_0>2$. From the definition of $S(h)$ it follows
\begin{equation} \label{eq3}
c_0=\frac{1}{p}\left(\sum_{h=A_0p+(\rho+\epsilon)p}^{A_0p+p}(1-\frac{h}{pA_1})\right)=
(1-\rho-\epsilon)\left(1-
\left(\frac{(1+\rho+\epsilon)+2A_0}{2A_1}\right)\right)\,.
\end{equation}
The condition $c_0\, m_0>2$ will imply, as we will see below
in Proposition \ref{noextingue}, that
the population does not extinguish, that is, it has the permanence property
(see Definition \ref{permanente}).

\item The exponent $\gamma$ satisfies $\gamma> 1$.
\end{enumerate}

The following proposition shows that a dynamical system governed by equation
(\ref{clave})
 and respecting the restrictions $1.$ to $7.$ above is bounded.
\begin{Prop} \label{acotado}
For all $t=1,2,\ldots ,A_0p$ we have $N_t\leq N_{max}:=
m_0\left(\frac{(A_1-A_0)^2}{2A_1}\right)$.

\end{Prop}
\begin{proof}
Since $\gamma> 1$, inequalities (\ref{eq2}) imply $N_jm(N_j)\leq
m_0$ for all $j$. Moreover from $m_\rho\leq 1$ we obtain
$$N_t=\sum_{h=A_0p}^{A_1p-1}N(t-h)
m(N(t-h)) m_\rho(t-h) S(h) \Delta h\leq $$
 $$\sum_{h=A_0p}^{A_1p}m_0
S(h) \frac{1}{p}=\frac{m_0}{p}\sum_{h=A_0p}^{A_1p} S(h)=
\frac{m_0}{p}\sum_{h=A_0p}^{A_1p}(1-\frac{h}{pA_1})=$$
$$\frac{m_0}{p}\left((A_1-A_0)p-\frac{A_1p(A_1p+1)-A_0p(A_0p+1)}{2A_1p}\right)\leq$$
$$m_0\left((A_1-A_0)-\frac{A_1(A_1+1/p)-A_0(A_0+1/p)}{2A_1}\right)\leq$$
$$m_0\left((A_1-A_0)-\frac{A_1^2-A_0^2}{2A_1}\right)
= m_0\left(\frac{(A_1-A_0)^2}{2A_1}\right)\, .$$
\end{proof}
By induction we obtain that for all $t\geq 0$, $N_t\leq N_{max}$.
\begin{Obs}
For the values $A_1=2$, $A_0=0.18$, $m_0=50$ we have $N_{max}\approx
41.4\,$.
\end{Obs}
\subsection{Permanence.}
In this section we verify that the population given by equation (\ref{clave})
 and respecting the restrictions $1.$ to $7.$ above, in particular conditions (\ref{eq2}) and (\ref{eq3}), does not extinguish.
\begin{Df} \label{permanente}
We say that a system $P(t)$ modeling the evolution of a population is {\em
permanent}, or satisfies the {\em permanence} property, if for any
positive initial vector value $P_0$, there is
$\epsilon>0$ such that the solution $P(t)$ satisfies $$\liminf_{t\geq
0}P(t)\geq \epsilon\, .$$
\end{Df}
If a given system is permanent then, assuming that the environmental conditions
do not change
in time, the associated population will not extinguish.
Thus, concerning with population dynamics this property is very important.

 The next proposition shows that the system under study satisfies the permanence property.

\begin{Prop} \label{noextingue}
If $i(N)=\min\{N_t,\, t\in [-pA_1,0]\}>0$ then $N_t>0$ for all
$t=0,1,\ldots ,A_0p$. Moreover,
\begin{itemize}
\item If $i(N)\leq N_{max}^{1-\gamma}$ then $N_t\geq
\frac{c_0\,m_0}{2} i(N)>i(N)$, $t\in [0,pA_0]$.
\item If $i(N)\geq N_{max}^{1-\gamma}$ then $N_t\geq \frac{c_0\,m_0}{2}
N_{max}^{1-\gamma}$, $t\in [0,pA_0]$.
\end{itemize}
\end{Prop}
\begin{proof}
If $N(t-h)\leq 1$ then, by (\ref{eq2}), $$N(t-h)m(N(t-h))\geq
N(t-h)\frac{m_0}{2}\geq i(N)\frac{m_0}{2}\, .$$ Otherwise $N(t-h)>1$
and then, again by (\ref{eq2}), $$N(t-h)m(N(t-h))\geq
\min\{N(t-h)^{1-\gamma}\frac{m_0}{2}, N(t-h)\frac{m_0}{2}\}\geq
\min\{N_{max}^{1-\gamma}\frac{m_0}{2}, i(N)\frac{m_0}{2}\}\, .$$
Hence we have
$$N_t=\sum_{h=A_0p}^{2p-1}N(t-h)
m(N(t-h)) m_\rho(t-h) S(h) \Delta h\geq $$
$$\frac{1}{p}\min\{N_{max}^{1-\gamma}\frac{m_0}{2},
i(N)\frac{m_0}{2}\}\sum_{h=A_0p}^{2p-1}m_\rho(t-h) S(h)\geq$$
$$\min\left\{N_{max}^{1-\gamma}\frac{m_0}{2},
i(N)\frac{m_0}{2}\right\}\,\frac{1}{p}\left(\sum_{h=A_0p+(\rho+\epsilon)p}^{A_0p+p}
(1-\frac{h}{pA_1})\right)=
\min\left\{N_{max}^{1-\gamma}\frac{c_0m_0}{2},
i(N)\frac{c_0m_0}{2}\right\}\, .$$
 Since, by (\ref{eq3}), $c_0\,m_0>2$ we get by induction that $N(t)>0$ for all $t\in[ 0,A_0p]$.
 \end{proof}
 Clearly Proposition \ref{noextingue} implies that $N_t>0$ for all $t\geq 0$.
 \begin{Cor} \label{todo t}
 There is $t_0>0$, depending on the initial vector value, such that we have
$N(t)\geq \frac{c_0m_0}{2}N_{max}^{1-\gamma}$, $t\geq t_0$.
\end{Cor}

\subsection{Existence of fixed points.}

The following corollary is a straightforward consequence of Propositions
\ref{acotado} and \ref{noextingue}.
\begin{Cor} \label{compacto}
 If $N_t>0$ for all $t\in[-A_p,0]$ then there is $t_0>0$ such that $\frac{c_0m_0}{2}N_{max}^{1-\gamma}\leq
N_t\leq N_{max}$ for $t\geq t_0$. In particular $T$ maps the compact
set
$$\mathcal{K}=\left[\frac{c_0m_0}{2}N_{max}^{1-\gamma},N_{max}\right]^{pA_1+1}\quad \mbox{
into itself}\, .$$
\end{Cor}
\qed

Now set \begin{equation} \label{hache} H:=\left\{N=(N_0,N_1,\ldots
,N_{2p})\in\R^{2p+1}\,:\, \forall\, j=0,1,\ldots
2p:\,N_j>0\right\}\, .\end{equation}
Observe that Proposition \ref{noextingue} together with Corollary \ref{todo t} imply that $T$ maps $H$ into itself.

Next we prove that $T:H\to H$ is Lipschitz.
\begin{Lem} \label{Lipschitz}
$T:H\to H$ is a Lipschitz function.
\end{Lem}
\begin{proof}
We put in $\R^{2p+1}$ the sup norm:
$\|x\|=\|(x_0,x_1,\ldots,x_{2p})\|=\sup_{t=0,\ldots ,2p}|x_j|$.

From the definition of $T$ we have $T(N_0,N_1,\ldots,
N_{2p})=(N_p,N_{p+1},\ldots, N_{3p})$. Hence for all $j=0,\ldots ,p$ we have
\begin{equation} \label{bimba}
\, |(T(N)-T(N'))_j|=|N_{j+p}-N'_{j+p}|\leq \|N-N'\|\,
.\end{equation}

For $j=p+1,\ldots ,2p$, the difference  $|N_{(t-h)}
m(N_{(t-h)})-N'_{(t-h)}m(N'_{(t-h)})|$ can be estimated as follows:
%\footnote{Here to simplify notation we write $N_{(t-h)}$ instead
%of $N_{t-h}$.}:
%for $t\geq 2p$
\begin{enumerate}
\item[(a)]
If $N_{(t-h)}\leq 1$ and $N'_{(t-h)}\leq 1$ then by
inequalities (\ref{eq2}) we have that $$|N_{(t-h)}
m(N_{(t-h)})-N'_{(t-h)}m(N'_{(t-h)})|\leq
m_0|N_{(t-h)}-N'_{(t-h)}|\, .$$
\item[(b)]
If $N_{(t-h)}\geq 1$ and $N'_{(t-h)}\geq 1$ then, again by
 (\ref{eq2}), we have that $$|N_{(t-h)} m(N_{(t-h)})-N'_{(t-h)}m(N'_{(t-h)})|\leq
|(N_{(t-h)})^{1-\gamma}-(N'_{(t-h)})^{1-\gamma}|m_0\, .$$
 By the Mean Value Theorem,  there is $\widetilde{N}\in (N_{(t-h)},N'_{(t-h)})$  such
 that
 $$|(N_{(t-h)})^{1-\gamma}-(N'_{(t-h)})^{1-\gamma}|=|1-\gamma|\widetilde
 N^{-\gamma}|N_{(t-h)}-N'_{(t-h)}|\, .$$
 Since $\gamma>1$ and $\widetilde N>1$ we obtain
 $$m_0|(N_{(t-h)})^{1-\gamma}-(N'_{(t-h)})^{1-\gamma}|\leq
 m_0(\gamma-1)|N_{(t-h)}-N'_{(t-h)}|\, .$$
\item[(c)]
If one of the above quantities is greater than 1 and the
other is not, say $N'_{(t-h)}>1$ and $N_{(t-h)}\leq 1$, then
$$|N_{(t-h)}
m(N_{(t-h)})-N'_{(t-h)}m(N'_{(t-h)})|=m_0|N_{(t-h)}-(N'_{(t-h)})^{1-\gamma}|\,
.$$ If $N_{(t-h)}\geq (N'_{(t-h)})^{1-\gamma}$ then, since
$0<N_{(t-h)}\leq 1$ and $1-\gamma<0$ we get
$$m_0|N_{(t-h)}-(N'_{(t-h)})^{1-\gamma}|=m_0(N_{(t-h)}-(N'_{(t-h)})^{1-\gamma})\leq$$
$$m_0((N_{(t-h)})^{1-\gamma}-(N'_{(t-h)})^{1-\gamma})=$$
$$m_0|(N_{(t-h)})^{1-\gamma}-(N'_{(t-h)})^{1-\gamma}|\leq m_0(\gamma-1)|N_{(t-h)}-N'_{(t-h)}|\,
.$$ Otherwise, if $N_{(t-h)} < (N'_{(t-h)})^{1-\gamma}$ then, since
$N'_{(t-h)}>1$ and $1-\gamma<0$, we have $0>N_{(t-h)} -
(N'_{(t-h)})^{1-\gamma}>N_{(t-h)}-N'_{(t-h)}$ and therefore
$$|N_{(t-h)}
m(N_{(t-h)})-N'_{(t-h)}m(N'_{(t-h)})|=
 m_0|N_{(t-h)}-(N'_{(t-h)})^{1-\gamma}|\leq m_0|N_{(t-h)}-N'_{(t-h)}|\, .$$
\end{enumerate}
Next, to estimate $|N_t-N'_t|$ for $t=p,p+1,\ldots ,A_0p$, we
 use (a), (b) and (c) above as below. Let $L=\max\{m_0,m_0(\gamma-1)\}$. Taking
into account that $m_\rho(t-h)$ and $S(h)$ are between 0 and 1 and
$\Delta h=\frac{1}{p}$ we obtain that:
\[|N_t-N'_t|=\left|\sum_{h=A_0p}^{2p-1}N_{(t-h)} m(N_{(t-h)}) m_\rho(t-h) S(h) \Delta
h-\right. \left.\sum_{h=A_0p}^{2p-1}N'_{(t-h)} m(N'_{(t-h)})
m_\rho(t-h) S(h) \Delta h\right|=\]
$$\left|\frac{1}{p}\sum_{h=A_0p}^{2p-1}\big(N_{(t-h)}
m(N_{(t-h)})-N'_{(t-h)}m(N'_{(t-h)})\big) m_\rho(t-h) S(h)
\right|\leq$$
$$\frac{1}{p}\sum_{h=A_0p}^{2p-1}\big|N_{(t-h)}
m(N_{(t-h)})-N'_{(t-h)}m(N'_{(t-h)})\big| m_\rho(t-h) S(h) \leq$$
\begin{equation} \label{pimba} \frac{1}{p}\sum_{h=A_0p}^{2p-1}L|N_{(t-h)-}N'_{(t-h)}|\leq
\frac{1}{p}\sum_{h=A_0p}^{2p-1}L\max_h|N_{(t-h)}-N'_{(t-h)}|\leq
(A_1-A_0)L\|N-N'\|\, .\end{equation}

Taking into account that $1\leq (A_1-A_0)L$ and inequalities
(\ref{bimba}) and (\ref{pimba}) we have that for all $j=0,\ldots
,p,p+1,\ldots ,p+A_0p$, $|T(N)-T(N')|_j\leq
(A_1-A_0)L\cdot\|N-N'\|$. By induction, since $A_0p>1$, we obtain that
 $\|T(N)-T(N')\|\leq (A_1-A_0)L\cdot\|N-N'\|$ finishing the proof.
\end{proof}

\begin{Cor} \label{fijo}
There is a fixed point $p$ for $T:\mathcal{K}\to \mathcal{K}$.
\end{Cor}
\begin{proof}
By Lemma \ref{Lipschitz} the map $T$ is Lipschitz hence continuous. Moreover
$\mathcal{K}$ is a $(2p+1)$-dimensional topological disk. Hence
Brouwer Fixed Point Theorem applies, \cite[Chapter 4, Section 7]{Sp}.
\end{proof}

\begin{Obs}
Since every two years ($A_1=2$) the rodent population is renewed
perhaps it is more natural to search for fixed points for
$T^2:\mathcal{K}\to\mathcal{K}$. So, we are interested in both,
fixed points and period-two points $N\in\mathcal{K}$. Their
existence is guaranteed by Corollary \ref{fijo}.

 In Appendix E we estimate the coordinates of a fixed point $p$ of $T^2:H\to H$. We find that the distance given by the norm of the supremum between
$p$ and $T^2(p)$ is about $8.0148\times 10^{-14}$ and the $l^1$ norm
is about $4.0353\times 10^{-12}$. This estimate of $p$
is better than that obtained by Arlot,
\cite[Section B.8]{Ar}, which is of order $10^{-4}$ for the $l^1$ norm.
\end{Obs}
\section{Existence of an attractor for the discrete model.}

\begin{Prop} \label{atractor}
Let $\Lambda=\bigcap_{n\geq 0}T^n(\mathcal{K})$. Then
$\Lambda\neq\emptyset$ is compact $T$-invariant and there is a
neighborhood $U=U(\Lambda)$ such that $T(\overline{U})\subset U$,
i.e., $\Lambda$ is an attractor for $T$.
\end{Prop}
\begin{Obs}
We are not assuming that $\Lambda$ is transitive in the definition
of attractor.
\end{Obs}
\begin{proof}
 Since $T(\mathcal{K})\subset \mathcal{K}$
we have that $C_n=\cap_{j= 0}^{n}T^j(\mathcal{K})$ is a decreasing
sequence of non empty compact subsets of $\R^{2p+1}$; $C_0\supset
C_1\supset\cdots\supset C_n\supset\cdots$. Thus, by Baire Theorem, we have
that $\Lambda\neq\emptyset$ and $\Lambda$ is compact.

 By definition of $\Lambda$
we have
$$T(\Lambda)=T(\cap_{n\geq
0}T^n(\mathcal{K}))\subset\cap_{n\geq 0}T^{n+1}(\mathcal{K})\subset
\cap_{n\geq 0}T^n(\mathcal{K})=\Lambda\, , $$ proving that $\Lambda$
is $T$-invariant.

Let $[\mathcal{K}]_\epsilon:=\{N\in\R^{2p+1}\,: \,
\dist(N,\mathcal{K})\leq \epsilon\}$ and $\epsilon>0$ be so
small that $$[\mathcal{K}]_\epsilon\subset H=\{N=(N_0,N_1,\ldots
,N_{2p})\in\R^{2p+1}\,:\, \forall\, j=0,1,\ldots 2p:\,N_j>0\}\, .$$
By Proposition \ref{noextingue} and Proposition \ref{compacto}, for all
$x\in[\mathcal{K}]_\epsilon$ there is $n(x)>0$ such that
$T^{n(x)}(x)\in \mathcal{K}$. By continuity of $T$, see Lemma
\ref{Lipschitz}, there is $U(x)$ a neighborhood of $x$ contained in
$H$ such that  $T^{n(x)}(y)\in\mathcal{K}$ for all $y\in U(x)$. By
compactness of $[\mathcal{K}]_\epsilon$ there is $n_1>0$ such that
 $T^n([\mathcal{K}]_\epsilon)\subset
\mathcal{K}$ for all  $n\geq n_1$.

Let now  $U(\Lambda)$ be a neighborhood of $\Lambda$ contained in
$[\mathcal{K}]_\epsilon$.
\begin{claim}
There is $n_0>0$ such that
 $T^n(\mathcal{K})\subset U(\Lambda)$ for all $n\geq n_0$.
 \end{claim}
 \begin{proof}
 The proof goes by contradiction. If
it were not true, for all $j\in\N$ there would exist
$x_j\in\mathcal{K}$ and $n_j>n_{j-1}$, such that $T^{n_j}(x_j)\notin
U(\Lambda)$. Since $\mathcal{K}$ is compact there exists a
convergent subsequence from $\{T^{n_j}(x_j)\}_{j\in\N}$. Without
loss we may assume that $\{T^{n_j}(x_j)\}_{j\in\N}$
itself converges to a point $z\in\mathcal{K}$. Such a  point $z$ cannot
be in $\Lambda$ since $\,T^{n_j}(x_j)\notin
U(\Lambda)$ for every $ j\in\N $. But, since $\,T^{n+1}(\mathcal{K})\subset T^n(\mathcal{K})$ for all $
n\in\N,$ we obtain
$T^{n_j}(x_j)\in \cap_{h= 0}^{n_j}T^h(\mathcal{K})$. Moreover, $z\in
\cap_{h= 0}^{n_j}T^h(\mathcal{K})$, otherwise there is $\epsilon>0$
such that $\dist (z,\cap_{h= 0}^{n_j}T^h(\mathcal{K}))>\epsilon$.
But  $\cap_{h= 0}^{n_j}T^h(\mathcal{K})\supset
\cap_{h= 0}^{n_{j+1}}T^h(\mathcal{K})$ for all $j\in\N$, so that
$\dist(z,\cap_{h= 0}^{n_{j+l}}T^h(\mathcal{K}))\geq \epsilon$ for every  $l\geq 0$
contradicting the fact that $T^{n_{j+l}}(x_{j+l})\to z$ when
$l\to\infty$. It follows that
$T^n(\mathcal{K})\subset U(\Lambda)$ for all $n\geq n_0$, proving the claim.
\end{proof}

To conclude the proof of the proposition it is enough to verify that there is
 $n_2>0$ such that
 $T^{n_2}(\overline{U(\Lambda)})\subset U(\Lambda)$. This follows
 from the fact that $T^{n_0}(\mathcal{K})\subset U(\Lambda)
 \subset\overline{U(\Lambda)}\subset
 [\mathcal{K}]_\epsilon$ taking $n_2=n_0+n_1$, thus $\Lambda$ is an attractor.
\end{proof}

It is clear that the fixed point $p$ given by Corollary \ref{fijo}
belongs to $\Lambda$. In \cite[Section B.8]{Ar} by numerical methods
it is found a candidate to be a fixed point. As we have pointed out above, in view of Corollary
\ref{fijo}, the search for such a fixed point has sense.

\begin{Obs}
By Corollary \ref{compacto}, the basin of attraction of $\Lambda$
is the whole set $H$  of points with positive coordinates (see equation {\rm (\ref{hache})} ).
Moreover,  since $\mathcal{K}$ is a disk, we can choose $U(\Lambda)$ simply connected in the proof of Proposition \ref{atractor}. These facts have some theoretical implications that
we discuss in section \ref{tom mix}.
\end{Obs}
%=\{N=(N_0,N_1,\ldots ,N_{2p})\in\R^{2p+1}\,:\,
%\forall\, j=0,1,\ldots 2p:\,N_j>0\}\, .$$

\begin{Lem} \label{nosingular}
Assume that $S(2p-1)> 0$ and $m_\rho(1)=1$. Moreover also assume that $T$ depends smoothly
on $N=(N_0,N_1,\ldots , N_{2p})$ and that $\frac{\partial(
N_jm(N_j))}{\partial N_j}\neq 0$, $j=0,1,\ldots ,p$. Then
the differential $D_NT:\R^{2p+1}\to\R^{2p+1}$ is a
non singular linear map.
\end{Lem}
\begin{proof}
Let us duplicate the $(p+1)$-th coordinate, $N_p$, of $N=(N_0,\ldots,
N_p,\ldots ,N_{2p})$, i.e., we write $\widehat N=(N_0,\ldots
N_p,N_p,\ldots N_{2p})=(N_{(0)},N_{(1)})$, and consider $\widehat
T(N_{(0)},N_{(1)})=(N_{(1)},N_{(2)})$ where
$N_{(2)}=(N_{2p},\ldots,N_{3p})$. Thus, since the
$p^{th}$-coordinate equals the $(p+1)^{th}$-coordinate, $\widehat
T(N_{(0)},N_{(1)})$ is such that $\Pi_p(\widehat
T(N_{(0)},N_{(1)})=T(N)$ and if $\widehat T$ is locally injective then
$T$ is locally injective too. Here $\Pi_p:\R^{2p+2}\to\R^{2p+1}$ is
the projection
$$\Pi_p(x_0,\ldots,x_p,x_{p+1},\ldots,x_{2p+1})=(x_0,\ldots,x_{p-1},x_{p+1},\ldots,x_{2p+1})\, .$$

Taking into account that $(N_{2p},\ldots,N_{3p})$ depends on
$(N_0,\ldots,N_p,\ldots,N_{2p})$, this artifice allows us to write
$\widehat T(N_{(0)},N_{(1)})=(N_{(1)},F(N_{(0)},N_{(1)}))$, and
therefore
$$D\widehat T=\left(
            \begin{array}{ccc}
              A & | & Id \\
              ----&----&----  \\
              \frac{\partial F}{\partial N_{(0)}} & | & \frac{\partial F}{\partial N_{(1)}} \\
            \end{array}
          \right)
$$ where $A$ is a $(p+1)\times (p+1)$ matrix of the form
$$A=\left(
    \begin{array}{cccc}
      0 & 0 & \cdots\,0 & 1 \\
      0 & 0 & \cdots\,0 & 0 \\
      \cdots & \cdots & \cdots & \cdots \\
      0 & 0 & \cdots\,0 & 0 \\
    \end{array}
  \right)
$$
and $Id$ is the identity $(p+1)\times(p+1)$ matrix.

To prove that
$\widehat T$ is locally injective it suffice to prove that $\det D\widehat
T\neq 0$. Hence, since $\det(A)=0$ we are left to prove that
$\det\left(\frac{\partial F}{\partial N_{(0)}}\right)\neq 0$. For
this we proceed as follows. Using the expression for $N_t$ given at
equation (\ref{clave}) and denoting $\frac{\partial
(N_jm(N_j))}{\partial N_j}$ by $h(N_j)$ we compute $\frac{\partial
F}{\partial N_{(0)}}$ and find
$$\frac{1}{p}\left(
\begin{array}{cccc}
h(N_0)m_{\rho}(1)S(2p-\!\!1) & h(N_1)m_\rho(2)S(2p-\!\!2) & \ldots\ldots\ldots & h(N_p)m_\rho(p)S(p) \\
0 & h(N_1)m_\rho(1)S(2p-\!\!1) & \ldots\ldots\ldots & h(N_p)m_\rho(p-1)S(p+\!\!1) \\
0 & 0 & \ldots\ldots\ldots & \ldots\ldots\ldots \\
\cdots\cdots\cdots & \cdots\cdots\cdots & \cdots\cdots\cdots & \cdots\cdots\cdots \\
0 & 0 & \ldots\ldots\ldots  & h(N_p)m_\rho(1)S(2p-\!1) \\
\end{array}
\right)\,.
$$
 Since by hypothesis $h(N_j)\neq 0$ the thesis follows.
\end{proof}

%By Lemma \ref{nosingular} we have that $DT$ is regular. Hence we get
\begin{Cor} \label{inyectivo}
Under the hypothesis of Lemma \ref{nosingular} we have that
$T:\Lambda\to\Lambda$ is locally injective.
\end{Cor}

\begin{Obs} \label{noglobal}
Albeit $T:H\to H$ is locally injective, by Lemma \ref{nosingular}, it is not globally injective.
To see this assume that $N_{max}>1$, $\gamma>1$ and that the definition of $m(N)$ is given by
equation {\rm(\ref{emedeene})}. If $(N_0,N_1,\ldots
,N_{2p-1})=(N_{max},N_{max},\ldots ,N_{max})$ then for $t=2p\,,\, 2p+1,\ldots, 2p+A_0p$ we get
$$N_t =\frac{1}{p}\sum_{h=A_0\,p}^{2\,p-1}N_{t-h} m(N_{t-h}) m_{\rho}(t-h) S(h)\,
 =\frac{1}{p}\sum_{h=A_0\,p}^{2\,p-1}N_{max} m(N_{max})
m_{\rho}(t-h) S(h)\,$$
\begin{equation} \label{eqnoglobal}
 =\frac{1}{p}\,N_{max}^{1-\gamma}\, m_0 \sum_{h=A_0\,p}^{2\,p-1}
m_{\rho}(t-h) S(h)\, . \end{equation}

 Similarly if we put $(N_0,N_1,\ldots
,N_{2p-1})=(N_{max}^{1-\gamma},N_{max}^{1-\gamma},\ldots ,N_{max}^{1-\gamma})$ we obtain the same values for $N_t$. By induction we get that all values are the same for $t\geq 2p$ implying that $T$ is not globally injective.
% On the other hand none of these points are in $\Lambda$.
\end{Obs}

Let us point out that:
\begin{enumerate}
\item
In the original model, \cite{Ar}, $m(N_j)$ is given by equation
(\ref{emedeene}) i.e., $m(N_j)=m_0$ if $N_j\leq 1$ and
$m(N_j)=m_0N^{-\gamma}$ if $N_j>1$. Hence $N_jm(N_j)=m_0N_j$ if
$N_j\leq 1$ and $N_jm(N_j)=m_0N_j^{1-\gamma}$ if $N_j>1$ implying
that
$$h(N_j)=\frac{\partial (N_jm(N_j))}{\partial N_j}=
\left\{\begin{array}{l}
          m_0\;\mbox{ if}\quad N_j\leq 1 \\
          m_0(1-\gamma)N_j^{-\gamma}\;\mbox{
          if}\quad N_j> 1\,
          \end{array}\right. \, .$$
          Since $\gamma>1$, we have $h(N_j)\neq 0$ for all $N_j\neq 1$.

\item Assuming  that $T^2$ is $C^1$,
Lemma \ref{nosingular} gives that the fixed point $p$ found at
Corollary \ref{fijo} has all its eigenvalues different from zero.
 The numerical approximation of the eigenvalues of $DT^2_p$, for the estimated value of $p$ obtained by \cite[Section 4.2.7]{Ar} and our own estimates
 %(identified with that given by Corollary \ref{fijo})
  gives that
%assuming that the approximation is reasonably accurate,
there is a single eigenvalue of modulus greater than 1 which is negative,
 and there are $A_1p$ eigenvalues of modulus less than 1. Hence $p$ is a codimension one
 hyperbolic fixed point of $T^2$.
\item
The hypothesis $S(2p-1)\neq 0$ is reasonable: otherwise one can see
that  for two initial vectors $N=(N_0, N_1,\cdots, N_{2p})$ and
$N'=(N'_0, N_1,\cdots, N_{2p})$ with
 $N_0\neq N'_0$ we get
$T(N)=T(N')$.
%In that case $T$ is not (locally) injective.
Thus the
 number $N_{A_1p}$ of individuals at time $A_1 p$ is not affected by
the first set $N_0$ of initial individuals. In another words the system looses
memory for  a number of years less than $A_1$ and so the actual dimension of the domain of $T$ would be less than $A_1p+1$.
\end{enumerate}

%Let us call $N^{(1)}$ to a preimage of $N$, $N^{(2)}$ to a second
%preimage of $N$, i.e.: a preimage of $N^{(1)}$ etc, and suppose that
%$(N_{1},\ldots ,N_{2p})=(N'_1,\ldots, N'_{2p})$ and $N_0\neq N'_0$.
%Then $N^{(1)}_p\neq N'^{(1)}_p$ and $N^{(1)}_j=N'^{(1)}_j$ for all
%$j=p+1,\ldots , 2p$.

\section{Study of $\Lambda$ for $(A_0,\rho,\gamma)=(0.18, 0.30,8.25)$.}

In what follows we will assume that $T$ is smooth ( see
\cite[Section 2]{Ar}) and that the calculations made for the parameter values
 $(0.18,0.30,8.25)$
%in \cite[Section 4.2.7]{Ar}
are accurate enough to obtain that
 if $p$ is the fixed point given by Corollary \ref{fijo} then the
eigenvalues $\lambda_1,\lambda_2,\ldots , \lambda_{2p}$ and $\mu$ of
$D_pT$ satisfies $|\lambda_j|<1$ for every $j=1,\ldots , 2p$, and
$\mu\approx -3.335$, in particular $|\mu|>1$\footnote{Arlot in \cite[Section 4.2.7]{Ar},
obtains that $\mu\approx -2,29$ for the parameter values $(0.15,0.30,8.25)$.}.
%Note that in \cite{Ar} it is not proved that
%$\lambda_j\neq 0$ for every $j$ although this fact is assumed there, see
%\cite[B10,B11]{Ar}.
Lemma \ref{nosingular} proves that $p$ is in
fact a hyperbolic fixed point with  $W^s(p)$ being a codimension one
 manifold and $W^u(p)$ an arc. Moreover, since $\Lambda$ is an attractor, we have that
 $W^u(p)\subset \Lambda$ from which the fractal dimension of $\Lambda$
  is strictly greater or equal than $1$. The calculations made in \cite[Section 4.2.5]{Ar} give for this fractal dimension a value around $1.33$ from which Arlot conjectures that locally the attractor is the product of a line by a Cantor set.

\label{tom mix}
Here we shall discuss if for the choice $A_0=0.18$, $\rho=0.30$ and
$\gamma=8.25$ the system given by $T$ can be transitive.\footnote{We
thank Enrique Pujals for fruitful discussions on this topic.}
\begin{Df} \label{topmix}
Let $f:X\to X$ be a continuous map defined in the topological space
$X$. We say that the system defined by $f$ is (topologically)
transitive if for every pair of non-empty open subsets $A,\, B$ of
$X$ there is $n\in\Z $ such that $f^n(A)\cap B\neq \emptyset$. The
dynamical system defined by $f$ is topologically mixing if for every
pair of non-empty open subsets $A,\, B$ of $X$ there is $N>0$ such
that $ f^n(A)\cap B\neq\emptyset$ for all $ n\geq N \,$.
\end{Df}

In \cite[Section 5]{Ar} it is pointed out the interest in studying
the case where the parameters are $A_0=0.18$, $\rho=0.30$,
$\gamma=8.25$: it is because the numerical simulations indicates
that for this parameter choice $T_{|\Lambda}$ is transitive, see
\cite[Section 4.1.3, figure 12]{Ar}. Moreover, in \cite[Section
4.2.7, figures 34 and 35]{Ar} the geometry of the attractor
$\Lambda$ is depicted from the successive iterates of the local
unstable manifold of the fixed point $p$. This suggests that
$W^u(p)$ is dense in $\Lambda$. This was confirmed by the numerical simulations done by us, see figure \ref{figatra3}.
The next proposition shows that if
the orbit of a point in $W^u(p)$ is dense in $\Lambda$ then $T_{|\Lambda}$ is in fact
topologically mixing.

\begin{Prop} \label{toptrans}
Let us assume that there exists $x_0\in W^u(p)$ such that
$\mbox{clos}(\mbox{orbit}^+(x_0))=\Lambda$ that there exists a homoclinic point
$x$ for $p$ that we do do not have tangencies
 between the stable and unstable manifold of $p$ and that forward iterates by $T^2$
 of an unstable segment $s\subset W^u(p)$ has diameter bounded away from zero.
 Then $T:\Lambda\to\Lambda$ is
topologically mixing.
\end{Prop}
\begin{proof}
%Let $N_0\in\Lambda$ be such that
%$\overline{\mbox{orbit}(N_0)}=\Lambda$, and $p$ the hyperbolic fixed
%point mentioned above. There exists $n_0$ such that $T^{n_0}(N_0)$
%is in a neighborhood $U(p)$ of $p$ in which we may assume
%$C^1$-linearizing coordinates. There is no loss of generality
%assuming that $N_0$ itself is contained in $U(p)$. Let $I$ be a
%small arc through $N_0$, parallel to the local unstable manifold of
%$p$, $W^u_{loc}(p)$. We take $I$ such that it cuts transversally the
%local stable manifold $W^s_{loc}(p)$. By the Inclination Lemma,
%\cite{PM}, forward iterates of $I$ by $T$ will accumulate in
%$W^u(p)$.
Observe that by hypothesis we have in particular that $\mbox{clos}W^u(p)=\Lambda$.
 Let $A\neq
\emptyset$ and $B\neq\emptyset$ be open subsets of $\Lambda$, i.e.,
there are open subsets $\mathcal{A}$ and $\mathcal{B}$ of
$\R^{2p+1}$ such that $A=\mathcal{A}\cap\Lambda$ and
$B=\mathcal{B}\cap\Lambda$.
%Without loss of generality we may
%assume that $\mathcal{A}$ and $\mathcal{B}$ are small boxes.
We will prove that there exists $n_0$ such that for all $n\geq n_0$
we have $T^n(A)\cap B\neq\emptyset$ thus proving that $T$ is
topologically mixing. Since $W^u(p)$ is dense in $\Lambda$ there is
$n_2>0$ such that $T^{n_2}(x_0)\in \mathcal{A}$. Thus $W^u(p)$ cuts
$\mathcal{A}$ in an arc $s$ containing $T^{n_2}(x_0)$.  Since
$\mbox{orbit}(x_0)$ is dense in $\Lambda$ there exists $n_1>n_2$
such that $T^{n_1-n_2}(x_0)\in U(p)$ where $U(p)$ is a neighborhood of $p$
in which we may assume that we have $C^1$-linearizing coordinates,
and $T^{n_1-n_2}(s)$ contains
an arc $J$ which intersects transversally $W^s_{loc}(p)$, this follows from
the assumptions we have done. By the
Inclination Lemma, see \cite[Chapter 2, \S 7]{PM}, $T^n(J)$
$C^1$-approaches on compact segments of $W^u(p)$. Let $\nu>0$ be the radius of a ball
contained in $\mathcal{B}$. There is $n_0>n_1$ such that
$T^{n_0}(J)$ is $\nu/2$-dense in $\Lambda$ and hence $T^n(J)$ is
$\nu/2$-dense in $\Lambda$ for all $n>n_0$. Thus $T^n(J)$ cuts
$\mathcal{B}$ implying that
$T^n(\mathcal{A})\cap\mathcal{B}\neq\emptyset$ for $n\geq n_0$. But
since $W^u(p)\subset \Lambda$ ($\Lambda$ is an attractor) we
conclude that $T^n(A)\cap B\neq\emptyset$ for $n\geq n_0$ proving
that $T$ is topologically mixing.

\end{proof}
\begin{Obs}
%\begin{itemize}
%\item
Roughly speaking the above result means that for the parameter
values $A_0=0.18$, $\rho=0.30$ and $\gamma=8.25$, from the
topological viewpoint we have that all possible states
$(N_0,N_1\ldots,N_{2p})\in\Lambda$ are visited and so a chaotic
behavior should be expected. On the other hand, since there are
fixed points like $p$ in $\Lambda$ if $(N_0,N_1\ldots,N_{2p})$ is
very near $p$ in practice we will see the same behavior for large
periods of time seeming that the population of these rodents is in
equilibria. On the other hand the hypothesis we have assumed
seems to be rather strong.

\end{Obs}

%\begin{Obs}
%Is it possible to conclude for the continuous model that there is
%always a fixed point for $T$ or for $T^2$? Perhaps letting
%$p\to\infty$ we may conclude this.
%\end{Obs}

Another consequence of the density of the unstable manifold of $p$ in $\Lambda$ is the following (see also Remark \ref{parecedensa}).
\begin{Prop} \label{inyecglobal}
If $\mbox{clos}(W^u(p))=\Lambda$ then $T^2_{|\Lambda}:\Lambda\to\Lambda$ is injective.
\end{Prop}
\begin{proof}
Indeed, $T^2$ is injective when restricted to $W^u(p)$, for, if it were not true, there would exist $x,y\in W^u(p)$ such that $T^2(x)=T^2(y)$. But, since $T^2(p)=p$ it holds that $W^u(p)=\cup_{n\in\N}T^2(W^u_\varepsilon(p))$ where $W^u_\varepsilon(p)$ is the $\varepsilon$-local-unstable manifold of $p$. Thus there is $N>0$ such that $x,y\in T^{2N}(W^u_\varepsilon(p))$ and, hence, there is an arc $\gamma\subset W^u(p)$ with end points $x$ and $y$. Applying $T^2$ to $\gamma$ we find a closed loop $T^2(\gamma)$ contained in $W^u(p)$ which contradicts the fact that $W^u(p)$ is homeomorphic to $\R$.

Assume now that there are $x,y\in\Lambda$ such that $T^2(x)=T^2(y)$. Since $T^2$ is locally injective there is $r_1>0$ such that $y\notin B(x,r_1)$ where $T^2_{|B(x,r_1)}:B(x,r_1)\to H$ is a homeomorphism. There exists also $r_2>0$ such that $T^2_{|B(y,r_2)}:B(y,r_2)\to H$ is a homeomorphism. Hence we may find $V(x)\subset B(x,r_1)$ a neighborhood of $x$ and $V(y)\subset B(y,r_2)$ a neighborhood of $y$ such that $T^2(V(x))=T^2(V(y))$ .
Since, by assumption, $W^u(p)$ is dense in $\Lambda$, there is an arc $\gamma\subset W^u(p)$ such that has its end points  $x'\in V(x)$ and $y'\in V(y)$ such that $T^2(x')=T^2(y')$ contradicting that $W^u(p)$ is homeomorphic to $\R$.
\end{proof}

%The hypothesis we have assumed in Proposition \ref{toptrans} and Proposition \ref{inyecglobal} should be justified.
%Next we justify why the hypothesis assumed in Propositions \ref{toptrans} and \ref{inyecglobal} are reasonable.
We point out that the numerical simulations presented in the appendices justify that the hypothesis assumed in Propositions \ref{toptrans} and \ref{inyecglobal} are reasonable. Indeed we found:
\begin{enumerate}
\item
 If there is a homoclinic point we must have positive entropy. We estimate in Appendix A the order-2 Kolmogorov entropy of the attractor, \cite{Ta}, and found a positive value $\approx 0.75$.
 \item  The absence of
tangencies should be checked in a certain way, at least in a neighborhood of $p$. In algorithm ''homclin4'' presented in Appendix B, we compute the angle between the local unstable manifold $W^u_\epsilon(p)$ and the iterate $T^m(\ell)$, of an arc $\ell\subset W^u_\epsilon(p)$, for $m>0$ such that  $T^m(\ell)$ is near $p$, founding in all cases values close to $\pi$ or 0 radians, thus $W^u_\epsilon(p)$ and $T^m(\ell)$ are almost parallel.
\item That there is a point in $W^u(p)$ whose orbit is dense is a
rather strong assumption. But when we plot the image of the first 1000 iterates of a single point of the local unstable manifold $W^u_{loc}(p)$, projected into $\R^3$ we
roughly recover the image of $\Lambda$ obtained plotting all the sequences of points
pseudo-randomly generated, see Appendix F. Moreover, in all the simulations done in algorithm ''entropia3'' presented in Appendix B, we always obtain that if $N\neq N'$ then $T^2(N)\neq T^2(N')$, indicating that the hypothesis of the density of $W^u(p)$ in $\Lambda$ assumed in Propositions \ref{toptrans} and \ref{inyecglobal} is consistent.
\item That forward iterates of a non trivial segment $s\subset W^u(p)$ have
their diameters bounded away from zero also is rather strong. But again in all the simulations done, in particular in all runs of algorithm ''homclin4'', presented in Appendix B, we verify that this is the case.
\item Moreover, there are theoretical results that point out that in a
setting like that of this model, we cannot expect $T^2_{|\Lambda}$ to be $C^1$-robustly transitive.
Indeed, by construction the attractor $\Lambda$ is contained in a simply connected neighborhood $U\subset \R^{201}$. Then by a $C^1$-small perturbation we may create a sink (see \cite{RS} for instance) whose basin of attraction may contain (part of) $W^u(p)$ . Nevertheless, the type of perturbations we can perform  with $T$ is not arbitrary and so we cannot reject
 {\it a priori} that for certain parameter values (like  $(A_0,\rho,\gamma)=(0.18, 0.30,8.25)$) the system is transitive.

\end{enumerate}
In the following subsections we check numerically the hypothesis of
Propositions \ref{toptrans} and \ref{inyecglobal}.
% and discuss the implications of having transitivity in $\Lambda$.

\subsection{Estimation of the Kolmogorov Entropy of the Attractor}
As a first step to estimate the presence of chaos in $\Lambda$ is to verify that it has sensibility with respect to initial data.
%\subsubsection{Sensibility to initial conditions}
%Taking into account the remarks made above
%we now wish to test numerically up to what extent the attractor $\Lambda$ has a
%chaotic behavior.
 To do so we have made computer simulations
of the system given by (\ref{clave}) with the parameter values $(0.18,0.30,8.25)$.
%One of the first results we have obtained is
That $\Lambda$ presents
sensibility to initial conditions has been pointed out by Arlot,
\cite[Section 4.2.6]{Ar}. To test this property we proceed as
follows:
\begin{enumerate}
\item
We generate $M$ independent initial vectors
$N^{(j)}=(N_0^{(j)},N_1^{(j)},\ldots, N_{A_1p}^{(j)})$, $1\leq j\leq
M$. In fact what we have done is to generate $M=400$ files with
initial data chosen in a pseudo-random way. We assume that these
$400$ initial data are independent.
\item We iterate $\ell$-times by $T^2$ so that $T^{2\ell}(N^{(j)})$
can be assumed, from the practical point of view, to belong to the
attractor. The value of $\ell$ that we have chosen is $\ell=10000$
so that we are considering $T^{20000}(N^{(j)})$. For simplicity of notation we
still denote this iterate by $N^{(j)}$.
\item We add a small noise $\Delta N^{(j)}$ to $N^{(j)}$ obtaining $\widetilde
N^{(j)}=N^{(j)}+\Delta N^{(j)}$. In the computer simulations we
choose $10^{-10}\leq \|\Delta N^{(j)}\|\leq10^{-8}$.
\item We specify a initial distance $d_0$ and compute for every $j$ the
integer $b_j$ such that
$$\|T^{2i}(N^{(j)}-T^{2i}(\widetilde N^{(j)})\|\leq d_0, \, 0\leq i< b_j,\quad\mbox{and}\quad
\|T^{2b_j}(N^{(j)}-T^{2b_j}(\widetilde N^{(j)})\|> d_0\, .$$ We
choose $d_0=0.1$ since we observe fast divergence between the orbits
when this distance is achieved.
\item In all the simulations we have done we find that $b_j\leq 80$. In
fact, we change the size of the perturbation finding that even with
$10^{-18}<\|\Delta N^{(j)}\|\leq 10^{-16}$, the value of $b_j$ satisfies $b_j\leq
200$. We conclude that there are numerical evidences that
$T_{|\Lambda}$ exhibits high sensibility to initial conditions.
\end{enumerate}

As a second step to test the chaotic behavior on $\Lambda$ we
estimate its order-2 Kolmogorov entropy $K$ giving by the average time for
two initially near orbits of the attractor to diverge.
More precisely, $K$ is calculated from the average time $t_0$ that is needed
for two points in the attractor, which are initially within a specified maximum
distance $d_0$, to separate until the distance between these points
has become larger than $d_0$.

 The Kolmogorov entropy of an attractor can be
considered as a measure for the rate of information loss along the
attractor or as a measure for the degree of predictability of points
along the attractor given an initial data. In general, a positive
entropy is considered as the conclusive proof that the dynamical
system is chaotic. A zero entropy represents a constant or a regular
phenomena that can be represented by a fixed point or a periodic
attractor, \cite{Ta}.

Here we apply the definitions of the order-2 Kolmogorov entropy
suggested by Takens in \cite{Ta} and by Grassberger and Procaccia
in \cite{GP}, see also \cite{GP2}. According to  these definitions, we will estimate the
entropy from the average time required for two nearby distinct
orbits of the attractor to diverge.

According to Takens \cite{Ta} and Grassberger and Procaccia \cite{GP},
the separation of distinct nearby orbits is assumed to be
exponential and the time interval $t_0$ required for two initially
nearby points to separate by a distance larger than $d_0$ will be
exponentially distributed according to
$$C(t_0)\sim e^{-Kt_0}, $$
where $K$ is the Kolmogorov entropy, see \cite{GP3}. For practical
purposes $C(t_0)$ may be transformed into a discrete distribution
function defined as
$$ C(b)=e^{-Kb\tau_s},\quad \mbox{with}\quad b=1,2,3,\ldots \, ,$$
where $\tau_s$ is the time step between two sampled data points.
Given an initial pair of independent points within a distance $d_0$,
the variable $b$ is the number of sequential pairs of points on the
attractor such that the interpoint distance is for the first time
bigger than $d_0$.

To estimate $K$ we proceed as follows.
\begin{enumerate}
\item
We generate $Z$ independent initial vectors
$N^{(j)}=(N_0^{(j)},N_1^{(j)},\ldots, N_{A_1p}^{(j)})$, $1\leq j\leq
Z$. For practical purposes we take for $Z$ the same $M=400$ files used to estimate sensibility to initial conditions.

%
% In fact we take $Z=400$ files with
%initial data chosen in a pseudo-random way used to test sensibility
%to initial conditions.
\item We iterate $\ell$-times by $T^2$ so that $T^{2\ell}(N^{(j)})$
can be assumed, from the practical point of view, to belong to the
attractor. The value of $\ell$ that we have chosen is $\ell=10000$
so that we are considering $T^{20000}(N^{(j)})\in\Lambda$. For
simplicity we still denote this iterate by $N^{(j)}$ and
will denote the initial $N^{(j)}$ by $T^{-20000}(N^{(j)})$,
but this is just a notation; we are not claiming that $T$ is globally invertible.

\item  For each $j=1,\ldots, Z$, we write in the file number $j$ the values of
$$T^{-20000}(N^{(j)}),\,N^{(j)},\,T^2(N^{(j)}),\,T^4(N^{(j)}),\ldots ,
T^{2044}(N^{(j)})\, .$$
\item  \label{above} Given a distance $d>0$, we search for pairs of vectors $T^{h_j}(N^{(j)})$,
$T^{h_i}(N^{(i)})$ such that
$\|T^{h_j}(N^{(j)})-T^{h_i}(N^{(i)})\|<d$.
According to \cite{STB} the value of $d$ should be smaller than $\frac{1}{100}$ of the absolute deviation $\delta_N$
The simulations we have
done give that the mean value $<N>$ of the population $N_t$ is about
$2.335$ and the average absolute deviation
$$\delta_N=\frac{1}{400\times 2046\times 200}\sum_{h,j,i}
\left|T^h(N^{j})_i-<N>\right|\approx 0.97\, ,$$ thus, we take  $d\leq 0.97\times 10^{-2}$ (the greater value of $d$ we have used is $d=\frac{1}{128}$).

\item Given $d$, $T^{h_j}(N^{(j)})$ and
$T^{h_i}(N^{(i)})$ as in item \ref{above}. above, we compute the
integer $b=b(i,j,h_i,h_j)$ such that
$$\|T^{2s}(T^{h_j}(N^{(j)}))-T^{2s}(T^{h_i}(N^{(i)}))\|\leq d, \quad
0\leq s< b,\quad\mbox{and}\quad$$
$$
\|T^{2b}(T^{h_j}(N^{(j)}))-T^{2b}(T^{h_i}(N^{(i)}))\|> d\, .$$
\item Letting $M=M(d)$ be equal to the
number of distinct pairs  $$T^{h_j}(N^{(j)}),\;
T^{h_i}(N^{(i)}),\;1\leq j< i\leq 400\, ,$$
verifying item
\ref{above}. we compute
 $\bar b=\frac{1}{M}\sum_{j=1}^M b_j$.
The program doing this task has to take care  to not duplicate the
number of times a given pair $T^{h_j}(N^{(j)}),\; T^{h_i}(N^{(i)})$
is computed and also to not consider as different strings the one
starting at $s=0$
$$\|T^{2s}(T^{h_j}(N^{(j)}))-T^{2s}(T^{h_i}(N^{(i)}))\|\leq d, \,
0\leq s< b,\quad\mbox{and}\quad$$
$$
\|T^{2b}(T^{h_j}(N^{(j)}))-T^{2b}(T^{h_i}(N^{(i)}))\|> d\, , $$ with
the sub-strings starting at $s=s_0>0$
$$\|T^{2s}(T^{h_j}(N^{(j)}))-T^{2s}(T^{h_i}(N^{(i)}))\|\leq d, \,
0<s_0\leq s< b,\quad\mbox{and}\quad$$
$$
\|T^{2b}(T^{h_j}(N^{(j)}))-T^{2b}(T^{h_i}(N^{(i)}))\|> d\, .$$

 \item Finally we estimate
the value of the entropy $K$ of $T^2$ by
$$ \widehat K =-\frac{1}{\tau_s}\ln\left|1-\frac{1}{\bar b}\right|\, , $$
where $\widehat K$ is the maximum-likelihood estimate of the entropy
$K$ (see \cite{STB}).
\item We repeat the items above for several values of $d$.
Taking $d\approx 1/100$ we find more than 2000 verifying item \ref{above}.,
 while for values of $d<1/50000$ the number of such pairs is too low, less than
100. More precisely, for $d=1/65536=0.0000152587890625$ we find 53
strings. This is reflected in the estimate of the standard deviation
of the entropy: for values of $d$ too small the sample is also small
and the estimation of $K$ is less accurate, as one can see in Appendix A.
\end{enumerate}

To test a confidence interval for the values obtained to $\widehat K$ we need to estimate its standard deviation. For this note that
the standard deviation of $\widehat K$ can
be obtained from the variance of $b$. To do so recall, \cite{STB}, that
$$var(b)=\frac{e^k}{(e^k-1)^2},\quad\mbox{where}\quad k=K\,\tau_s\,.$$
The standard deviation in the estimate of $\bar b$, computed in item 6. is given by
$$\sigma(\bar b)=\sqrt{var(b)/M}=\frac{e^{\widehat
k/2}}{\sqrt{M}(e^{\widehat k}-1)}\, .$$ For large values of $M$,
 $\sigma(\bar b)$ will be small. In that case we can use the
 derivative of the function $k=-\ln(1-1/b)$ in the point $\widehat k=
 \widehat K\tau_s$ to estimate the standard deviation of $k$.

\medbreak

The values obtained for the entropy of $T^2$ are listed in two
tables in Appendix A which contain also the values of $d$ we have used and
those of the standard deviation $\sigma_K$ of the entropy. For both
extreme values used for $d$, namely $d=1/128=0.0078125$ and
$d=1/65536=0.0000152587890625$ the results are less accurate, since
$0.0078125$ is ``too big'' with respect to $0.97\approx 1$, and for
$0.0000152587890625$ there are few sample points, see \cite{ER}.

Nevertheless all the estimates obtained show that $T^2_{|\Lambda}$ has
positive entropy, which implies that $T_{|\Lambda}$ also has positive
order-two entropy $K\approx 0.37$.

Thus we have strong numerical evidence that $\Lambda$ is a chaotic
attractor.

\begin{Obs}
We do not claim that we have estimated the entropy of $T_{|\Lambda}$.
The calculations made has to be seen as an indication that the model given
 by equation {\rm(\ref{clave})} exhibits a chaotic behavior.
 Rigorous proofs are needed to confirm our estimations.
\end{Obs}

\section{Existence of homoclinic points: numerical approach.}

In dynamical systems the presence of chaotic behavior
is often associated to the existence of homoclinic points.
We have assumed their existence in Proposition \ref{toptrans} to obtain that $\Lambda$ is topologically mixing.
 Next we check numerically their existence.
 To do it we proceed as follows:
 %\begin{enumerate}
 %\item
 \subsection{Approximated $W^u_{loc}(p)$.}
 Due to the fact that $W^u(p)$ is one dimensional a first attempt is to try to
pick a fundamental domain in $W^u_{loc}(p)$ and search by brute force if it is possible
to find a candidate to be a homoclinic point there.
Problem: we do not know precisely $W^u_{loc}(p)$. Moreover, the value of the fixed point $p$ is known only by an approximate value $\hat p$.
 But we know that there are only one eigenvalue $\mu$ of modulus greater than 1 of $DT^2_p$
 and $\mu$ is negative. Hence, since the other $A_1 p$ eigenvalues are small in modulus,
 in fact all of them have modulus less than $0.5$, we may assume that $W^s_{loc}(p)$ is a $A_1p$-dimensional disk so
 if we iterate $\hat p$ by $T^2$, since $\mu<0$ we have that the segment $[\hat p,T^2(\hat p)]$ cuts $W^s_{loc}(p)$ at a unique point.
By the $\lambda$-lemma
 we have that the successive iterates of  $[\hat p,T^2(\hat p)]$ by $T^2$ converges to
 $W^u_{loc}(p)$.

 Thus for numerical simulations we can take as $W^u_{loc}(p)$ one of these segments.
 In some of our simulations we choose
 $[T^{38}(\hat p), T^{40}(\hat p)]$ as $W^u_{loc}(p)$ and in others we take $W^u_{loc}(p)$ as
 $[T^{30}(\hat p),T^{32}(\hat p)]$.  Observe that the length of $[T^{38}(\hat p), T^{40}(\hat p)]$ is less than $10^{-3}$ and the length of   $[T^{30}(\hat p),T^{32}(\hat p)]$ is less than $10^{-4}$. Hence, since the mean value of the data is $2.335$ and that the absolute deviation is $0.97$ such lengths
 are relatively small.

 We subdivide the chosen segments in 10000 equal parts and iterate more than 2000 times by $T$
 every point $y$ of the subdivision finding the iterate $T^{2j}(y)$ closer to $\hat p$.
 In order to not consider misleading solutions, we discard the first 20 iterates and
 check that the orbit of $y$ is ``returning near the point $\hat p\,$'', i.e., we check
 that the minimum distance is not achieved in the $21^{th}$ iterate.
 Then we create a table containing the values of $y$ and of the iterate of $y$ closer to $\hat p$.
 Not that this procedure does not prove that any of such a point $y$ is a homoclinic point.
 %\item
 \subsection{Returning points.}
After this we find the value of $y_0$ and $j_0$ that minimizes $\dist(T^{2j}(y),\hat p)$.
In the simulations corresponding to $W^u_{loc}(p)\approx [T^{38}(\hat p), T^{40}(\hat p)]$
we find that
$$y_0=T^{38}(\hat p)+ \frac{5102}{10000}\left(T^{40}(\hat p)-T^{38}(\hat p)\right) \quad
\mbox{ and }\quad j_0=629\, .$$

 We find a suitable sub-interval $I_0$ such that
 $y_0\in I_0\subset I\subset [T^{38}(\hat p), T^{40}(\hat p)]$ ,
 we iterate 10 times by $T^2$ the point $y_0$ and the extreme points of the segment $I_0$,
 calling them $L_0$ and $R_0$
 \footnote{To try to subdivide the interval $I\subset [T^{38}(\hat p), T^{40}(\hat p)]$
 around $y_0$ of end points
$T^{38}(\hat p)+ (5101/10000)(T^{40}(\hat p)-T^{38}(\hat p))$ and
$T^{38}(\hat p)+ (5103/10000)(T^{40}(\hat p)-T^{38}(\hat p))$
 to obtain more precision is not a good idea since forward iterates
 by $T^2$ of $I$ increases their length exponentially fast.
 We loose any precision
 in the calculus after less than 20 iterations by $T^2$.}.
 After this we subdivide again $T^{20}(I_0)$ and find a small
 interval $I_1\subset T^{20}(I_0)$ around $T^{20}(y_0)$ and iterate again their end points
 $L_1,R_1$ and also $T^{20}(y_0)$. We continue with this procedure finding
 segments $I_h\subset T^{20}(I_{h-1})$  and their end-points $L_h, R_h$
 till we arrive to the value of $j_0$. There are cases that we cannot iterate $10$ times by $T^2$ because distances become relatively large or because we cannot assume $T^{20}([L_h,R_h])$ to be a straight segment and in that cases we reduce the step size.
The final step does not have to be a multiple of 10. We found that a suitable value
for the length of the initial segment $I_0$ is $ 1.122\times 10^{-7}$.
 To validate this procedure we have to check several things:
\begin{enumerate}
 \item  control that the length of $T^{20}(I_{h}) $ does not increase too much:
 we do not accept a length greater than $10^{-4}$. If the length of $T^{20}(I_{h}) $
 is greater than $10^{-4}$ we reduce the step used: first to 8 iterates by $T^2$
 and finally by $2$ iterates by $T^2$. In our computations
 we do not need to further reduce this number of iterates.
 \item  control that the segment $T^{20}(I_h)$ (or $T^{16}(I_h)$ or $T^4(I_h)$ in
 case that we have to choose a smaller step) does not bend too much: we require that $T^{20}(I_h)$
  behaves like a straight segment. To do so
  % we need not only to consider points
% $T^{20h}(y_0)$ (assuming, to simplify, that we are working with a multiple of 10)
 %and $L_h, R_h$, but
 we subdivide the segment $I_h$ into four equal smaller segments
 $[L_h,L'_h]$, $[L'_h,T^{20h}(y_0)]$, $[T^{20h}(y_0),R'_h]$, and
 $[R'_h,R_h]$. Next we  check that after 10 iterates of these intervals by $T^2$,
 the sum of their lengths
 satisfies that
 $$T^{20}([L_h,L'_h])+T^{20}([L'_h,T^{20h}(y_0)])+T^{20}([T^{20h}(y_0), R'_h])+
 T^{20}([R'_h,R_h])$$
 is almost the same as the length of $T^{20}([L_h,R_h])$.
We reject any case where the quotient between both quantities is greater than
$1.0001$, reducing the number of iterates if it were necessary\footnote{In fact at the scale we have chosen this has never been the case for reasonable values of $\ell([L_h,R_h])$.}.
\end{enumerate}

%\item
 \subsection{Far from tangencies.}
 After computing $T^{j_0}(y_0)$ and the corresponding points $L_{h_0}$ and $R_{h_0}$
for suitable $h_0$
\footnote{If the number of iterates is always 10 then we get $h_0=\left[\frac{j_0}{10}\right]$.}
we compute the angle between $[L_{h_0},R_{h_0}]$ and
$[T^{38}(\hat p), T^{40}(\hat p)]$. We expect to have an angle close to 0 or 180 degrees,
and in fact this is the case in all the simulations: we obtain for the angle the value of
 $3.108\times 10^{-5}$ radians.  This is an indication
that we are not near a tangency.

%\item
\subsection{Evidence of homoclinic points.}
\begin{enumerate}
\item
 For a suitable choice of $I_0=[L_0,R_0]$ we compute the angle between the segments
$[\hat p, L_{h_0}]$ and $[\hat p, R_{h_0}]$.
This is a key point in our calculations. Before we indicate how we proceed to do so, recall that  the codimension one submanifold $W^s_{loc}(p)$ of $\R^{A_1p+1}$ locally separates $\R^{A_1p+1}$ in two regions that we denote by $W^{s,+}$ and $W^{s,-}$.

On the one hand, if $L_{h_0}\in W^{s,+} $ and $R_{h_0}\in W^{s,-}$ then  $[L_{h_0},R_{h_0}]$ intersects $W^s_{loc}(p)$ and so we have a homoclinic point in this segment $[L_{h_0},R_{h_0}]$. Hence, by the $\lambda$-lemma the angle between successive iterates of
the vectors $[\hat p, L_{h_0}]$ and $[\hat p, R_{h_0}]$ would increase up to a value close to $\pi$.

%in case that we have a homoclinic
%point between $L_{h_0}$ and $R_{h_0}$ we should expect to have that this angle
%increase  to values close to 180 degrees when we iterate by $T^2$,

On the other hand, if both points are in the
same region with respect to $W^s_{loc}(p)$, say $L_{h_0} ,\, R_{h_0}\in W^{s,+}$,  then the segment $[L_{h_0},R_{h_0}]$ will not cut $W^s_{loc}(p)$ and, again by the $\lambda$-lemma, we have that the angle  between successive iterates of
the vectors $[\hat p, L_{h_0}]$ and $[\hat p, R_{h_0}]$ goes to
zero when we iterate by $T^2$. In this case the existence of a homoclinic point cannot be guaranteed.

In the simulations we have done, see Appendix F,
we obtain that for $I_0$ of length $ 1.122\times 10^{-7}$ the initial angle between $[\hat p, L_{h_0}]$ and $[\hat p, R_{h_0}]$ is
$1.337\mbox{ radians, approximately 77}$ degrees.  For the angle between
$[\hat p, T^2(L_{h_0})]$ and $[\hat p, T^2(R_{h_0})]$ we obtain a value of $3.011$
radians which is about 173 degrees.  For the angle between
$[\hat p, T^4(L_{h_0})]$ and $[\hat p, T^4(R_{h_0})]$ we obtain a value of $3.139$
radians which is about 180 degrees and for the angle between
$[\hat p, T^6(L_{h_0})]$ and $[\hat p, T^6(R_{h_0})]$ we obtain a value of $3.140$ radians.
For the subsequent iterates  the angle diminishes slightly but
up to the $14^{th}$ iterate we find that
the angle is close to $\pi$. Thus in that case we find evidence that a homoclinic point exists.

\item There are choices for
the length of $I_0$ that does not lead to such evidence. Due to the
exponential dilation in the unstable direction the behavior is rather sensible to this value.
If we choose $\ell(I_0)=1.046\times 10^{-7}$, instead of $ 1.122\times 10^{-7}$,
 we obtain at the final step that for this value both
$L_{h_0}$ and $R_{h_0}$ belong to the same local
connected component of $\R^{A_1p+1}\backslash W^s_{loc}(p)$.
In this case we have that the angle between $[\hat p, L_{h_0}]$ and $[\hat p, R_{h_0}]$ is
$1.358\times 10^{-3}\mbox{ radians}$, the angle between
$[\hat p, T^2(L_{h_0})]$ and $[\hat p, T^2(R_{h_0})]$ is $2.922\times 10^{-5}$ and the angle
between $[\hat p, T^4(L_{h_0})]$ and $[\hat p, T^4(R_{h_0})]$ is $4.582\times 10^{-6}$. This indicates that both points belong to the same region with respect to $W^s_{loc}(p)$. Thus we cannot ensure the existence of homoclinic points in this case.

  But as we have shown above, there are
  choices for the length of $I_0$, subject to all the mentioned restrictions,
  that render numerical evidence that
 we in fact do have a homoclinic point associated to the fixed point $p$.

\end{enumerate}

In the Appendix D we give the pseudo-code of the algorithms employed to
test the existence of homoclinic points.

In Appendix F we show the values of the approximate homoclinic point $y_0\in W^u_{loc}(p)$
and the angular values for the iterates $[\hat p, T^{2j}(L_{h_0})]$ and
$[\hat p, T^{2j}(R_{h_0})]$
for $j=0,1\ldots , 7$.

\subsection*{Acknowledgements}
M. J. Pacifico thanks Stefano Marmi who has introduced this problem to her. She also thanks the {\it Scuola Normale Superiore di Pisa} for its kind hospitality.

Jos\'e L. Vieitez thanks Universidad de Santiago de
Compostela, Spain,  UFRJ and IMPA, Rio de Janeiro, Brazil, for their kind
hospitality during part of the preparation of this article.

\section{Appendices.}
\subsection{Appendix A: numerical results for the entropy.} The following tables
gives the estimation of $K_{T^2}$ with
 $d$ varying from $d=1/128$  to $d=1/2048$ and $d$ varying from $d=1/4096$ to $d=1/65536$ respectively. The values of $d$ are evenly distributed.
 % in $[1/2048,\,1/128]$.

\medbreak
\center{
\begin{tabular}{|c|c|l|}
  \hline
  entropy estimated & standard deviation  & $d$ of the estimation \\
  \hline
  $0.71868973392930086$ & $0.019083966799452565$ & $0.0078125$ \\
  $0.72369632526302973$ & $0.01900037121106015$ &
  $0.00732421875$ \\
 $0.72434809650524738$ & $0.019137203419564205$ &
 $0.0068359375$\\
$0.71294980785612502$ & $0.019363360739997387$ &
$0.00634765625$\\
$0.72339347917167310$ & $0.019254622891041776$ &
$0.005859375$\\
$0.72850976773398998$ &  $0.019582964739043648$ & $0.00537109375$\\
$0.72335976581552290$ &  $0.019863263097474933$ & $0.0048828125$\\
$0.72087894452782916$ & $0.020143892470160681$ & $0.00439453125$ \\
$0.73055620248396048$ &  $0.020338205554372255$ & $0.00390625$ \\
$0.73804687146973547$ &  $0.020719763256481047$ & $0.00341796875$\\
$0.75191849797535868$ &  $0.021276865360116811$ & $0.0029296875$\\
$0.73746684855086799$ & $0.022038110478093715$ & $0.00244140625$\\
$0.74278320582790917$ & $0.022619630761394613$ & $0.01953125$\\
$0.74484051347505878$ & $0.023688490322879650$ & $0.00146484375$\\
$0.75258445595582577$ & $0.025936858832676262$ & $0.0009765625$\\
$0.77155883257911018$ & $0.030932730422267861$ & $0.00048828125$\\
\hline
\end{tabular}}

\medbreak \noindent
% The following table give the estimation of $K_{T^2}$ with
% $d$ varying from $d=1/4096$  to $d=1/65536$. The values of $d$ are evenly distributed in
% $[1/4096,\,1/65536]$.
%
%\bigbreak
\center{
\begin{tabular}{|c|c|l|}
  \hline
  entropy estimated & standard deviation  & $d$ of the estimation \\
  \hline
$0.79719104033302477$ & $0.038944604944862542$ & $0.000244140625$\\
$0.80273603232729273$ & $0.040161477427822390$ &
$0.0002288818359375$\\
$0.80026054504518599$ & $0.040945123242474202 $ &
$0.000213623046875$\\
 $0.80048080034198759$ &  $0.042277355365906191$ &
 $0.0001983642578125$\\
$0.78529717602382290$ & $0.043744279577575359$ &
$0.00018310546875$\\
$0.78083504948163185$ &  $0.045414270447522060$ &
$0.0001678466796875$\\
 $0.79617940645818047$ & $0.047304274588613750$ &
 $0.000152587890625$\\
 $0.80929969245707451$ & $0.049841776725230737$ &
 $0.0001373291015625$\\
$0.80683413899432532$ &  $0.053337136165239741$ &
$0.0001220703125$\\
$0.83598552255847603$ & $0.056494298338862138$ &
$0.0001068115234375$\\
$0.80841322520717467$ & $0.061957796245919238$ &
$0.000091552734375$\\
$0.86342039883772544$ & $0.068153726594861803$ &
$0.0000762939453125$\\
$0.85991632434641512$ & $0.076013458407264910$ &
$0.00006103515625$\\
$0.90006789636726776$ &  $0.091761409503824907$ &
$0.0000457763671875$\\
$0.79158725337319783$ &  $0.122667989144921260$ &
$0.000030517578125$\\
$0.96758402626170560$ &  $0.186693997202307350$ &
$0.0000152587890625$\\
\hline
\end{tabular}}

\subsection{Appendix B: description of algorithms.}

Taking into account \cite{TR}, we do not care so much about the embedding dimension and use directly
as vectors of data those given by
$N=(N_0,N_1,\ldots ,N_{200})$.

\begin{itemize}
\item
A first algorithm called ``ratones'' is used to generate 400 files
named datos$[i]$ $i=1,2,\ldots, 400$, each of which contains the
following data:
\begin{enumerate}
\item A random seed is generated to initialize a pseudo-random generator.
\item For each $i$ from 1 to 400 an initial vector of dimension 201 in which
every component is a real number $N_h$. This real number $N_h$ is in
fact a floating point number of 80 bits following IEEE
754-1985\footnote{IEEE Standard for Binary Floating-Point Arithmetic
(ANSI/IEEE Std 754-1985). Also known as IEC 60559:1989, Binary
floating-point arithmetic for microprocessor systems. } standards
for the representation, calculations and manipulations of real
numbers in a computer. The value of every element $N_h$ for $h=0$ to
$h=199$ is generated calling the RANDOM function available in the
Software Library. The value of $N_{200}$ is calculated from equation
(\ref{clave}). $N^{init}= (N_0,N_1,\ldots ,N_{200})$ is stored as
the first value in the corresponding file datos$[i]$.
\item From equation  (\ref{clave}) we compute the different values
of $N_h$ for $h\geq 201$, defining in this way recursively
$$T^2(N^{init}),\, T^4(N^{init}),\, T^6(N^{init}),\,\ldots\,.$$
    We discard the first $9999$ iterates and stored in datos$[i]$ the following $1024$ ones,
    $$T^{20000}(N^{init}).\,T^{20002}(N^{init}),\, \ldots, T^{22046}(N^{init})\, .$$
\end{enumerate}
\item A second algorithm that we call ``ratones1'' is used to perturb randomly $T^{20000}(N^{init})$ in each of the 400 files generated by ``ratones'' obtaining a vector $\widetilde N$.
    The random perturbations done vary from
    $-2^{-50}\approx -10^{-15}$ to $2^{-50}\approx 10^{-15}$ in each of
    the $h$-coordinates of $T^{20000}(N^{init})$ for $h$ from 0 to $199$.
    $\widetilde N_{200}$ is computed from
    equation (\ref{clave}).
\item The third algorithm we use, called ``sensible'', %es el entropia junto con ratones2
computes for each $i$ from 1 to 400 the number $b_i$ such that for
$j=0$ to $j=b_i-1$
$$\|T^j( T^{20000}(N^{init}))-T^j(\widetilde N)\|\leq 0.1\quad\mbox{and}
\quad  \|T^{b_i}( T^{20000}(N^{init}))-T^{b_i}(\widetilde N)\|>
0.1$$ We use the supremum norm in the calculations since this
accelerate the computations and it is clear that the results do not
depend on the norm used.

\item Algorithm, ``sensible'', also computes the mean value $<b>$ of $b_i$ as
$$<b>=\frac{1}{400}\sum_{i=1}^{400} b_i\, ,$$
in all the simulations done the value of $b_i$ was less than $180$
and $<b>\approx 100$.

\item The forth algorithm, ``dispersion'', %es el entropia
calculates the mean value $<N^{(i)}>$ of data stored in the files
datos$[i]$. It calculates also the mean value of all data which
gives a result of $<N>\approx 2.34$.
\item Algorithm ``dispersion'' also computes the absolute average
deviation $$\delta_N=\frac{1}{400\times 2046\times 200}\sum_{h,j,i}
\left|T^h(N^{j})_i-<N>\right|\approx 0.97\, .$$
\item Given a value $d>0$ the algorithm ``entropia3'' compares the data stored in
datos$[j]$ with that stored in datos$[i]$ discarding the initial
vectors (only after 20000 iterates by $T$ we assume that the vectors
are in $\Lambda$). For $1\leq j < i \leq
 400$ ``entropia3'' searches for pairs $T^{h_j}(N^{(j)}),\; T^{h_i}(N^{(i)})$ such that
 their distance, given by the norm of the supremum, is less $d$.
 ``entropia3'' runs 32 times generating 32 files named info$[k]$, $k=1,\ldots ,32$, of
records each of which contains
\begin{enumerate}
\item The number $i$ of file datos$[i]$,
\item the number of iterates $h_i$ by $T$ from $N^{(i)}$,
\item the value of  $T^{h_i}(N^{(i)})$,
\item the number $j$ of file datos$[j]$,
\item the number of iterates $h_j$ by $T$ from $N^{(j)}$,
\item the value of  $T^{h_j}(N^{(j)})$.
\end{enumerate}
For values of $d$ not so small we obtain huge files info$[k]$, and
as $d$ decreases the size of these files decreases. For
computational reasons we choose $d_{max}=1/128$ (corresponding to
info$[1]$ with $6,602\, KB$) and $d_{min}=1/65536$ (corresponding to
info$[32]$ with $196\, KB$). Of course the files info$[k]$ contain a
lot of redundant information since if $\dist(T^{h_j}(N^{(j)}),\;
T^{h_i}(N^{(i)}))<d$ and also $\dist(T^{h_j+l}(N^{(j)}),\;
T^{h_i+l}(N^{(i)}))<d$, with $l>0$ less than the least positive
value
 $b$ such that $\dist (T^{h_j+b}(N^{(j)}),T^{h_i+b}(N^{(i)}))\geq
 d$, we are storing
$(j,h_j,T^{h_j}(N^{(j)});i,h_i,T^{h_i}(N^{(i)}))$, and also
$(j,h_j+l,T^{h_j+l}(N^{(j)});i,h_i+l,T^{h_i+l}(N^{(i)}))$.
\item Finally the algorithm ``entropia4'' computes the estimation of the
 second order entropy, $\widehat K$, and its standard deviation
 using the information stored in the files info$[k]$ and the formulas given in \cite{STB}.

  For this we
 calculate for each
 $(j,h_j,T^{h_j}(N^{(j)});i,h_i,T^{h_i}(N^{(i)}))$ the least positive
value
 $b$ such that $$\dist (T^{h_j+b}(N^{(j)}),T^{h_i+b}(N^{(i)}))\geq
 d\,.$$
 In order not to duplicate information, once the value $b$ corresponding to
 $(j,h_j,T^{h_j}(N^{(j)});i,h_i,T^{h_i}(N^{(i)}))$  is calculated,
  we discard in this step the records $(j, h'_j,T^{h'_j}(N^{(j)});i,h'_i,T^{h'_i}(N^{(i)}))$
 such that $h_j+b\geq h'_j$ or $h_i+b\geq h'_i$ since these should have been
 taken into account in the previous step.
\end{itemize}

\begin{Obs} \label{parecedensa}
Although we have not taken care of the possibility that
$T^{20000}(N^{(i)})=T^{20000}(N^{(j)})$ with $i\neq j$, this (very
rare) possibility did not occurred in any of the simulations we have done.
Moreover, in accordance with Proposition \ref{inyecglobal}, in all these simulations, in particular in algorithm ''entropia3'', we always obtain that if $N\neq N'$ then $T^2(N)\neq T^2(N')$, so that the conjecture that $W^u(p)$ is dense in $\Lambda$ is not contradicted.
\end{Obs}

\subsection{Appendix C: Pseudo-code of the algorithms employed}

Here we give the pseudo code of the programs in a language close to
FreePascal, the style of programming is procedural.

\medbreak

{\footnotesize

  \noindent constants used\\
     A0= 0.18;
     p = 100;
     A1 = 2;
     gamma=8.25;
     m0=50;
     rho=0.30;
     pipa=1024;
     na=400;\\

\noindent type of data structures used is standard, in particular
``extended'' means a floating point number of 10 bytes and
``longint'' or ``integer'' means an integer number occupying 4 bytes
of memory according to the standards of IEEE. We also use arrays of extended or of integer and store
the data in sequential files of records.
%containing the relevant data.
%\newpage
\medbreak

\noindent {\bf function S($h$:integer):extended};\\

\{INPUT: $h\in\Z$\quad OUTPUT: $S(h)\in\R^+$\}

 begin\\
       $\quad$ if ($h<0$) or ($h>A1*p$) then S$:=$0
\quad                           else S$:=1-h/(A1*p+1)$\\
  end;

  \medbreak

  \noindent
  {\bf function mrho(h:integer):extended};

 \{ INPUT: $h\in\Z$\quad OUTPUT: $m_\rho(h)=\left[\begin{array}{l}
                                           1 \mbox{ if }  0\leq h\mod 1< \rho    \\
                                           0 \mbox{ elsewhere}
                                         \end{array}\right.
  $\}

  begin

        entrho:=trunc(rho*p);

        if ((h mod p) $<$entrho) then mrho:=0
                              else mrho:=1;

  end;

  \medbreak

{\bf function eme(N:extended):extended};

\{INPUT: $N\in\R^+$\quad OUTPUT: $m(N)\in\R^+$\}

  begin

       eme:=m0;
        lm:=N;

       if lm$>$1 then eme:=eme*lm**(-gamma)

  end;

\medbreak

{\bf procedure comienzoazar};

 begin

 randomize;
 semilla:=maxlongint;

end;

\medbreak

{\bf procedure AZAR(var n:longint)};

\{INPUT: random\_seed OUTPUT: pseudo-random number$\in\N$\}

begin

   x:=random(200000);
   n:=x;

    end;

\medbreak

{\bf function calculo(t:integer;ene:especial):extended};

\{INPUT: $t\in\Z$\quad OUTPUT: $N\in\R^{2A_1p+1}$\}

type

    especial = array[1..2*A1*p+1] of extended;

 begin

    lc:=0;

     for h:=floor(A0*p) to A1*p do
       begin

             lc:=lc+ene[t-h]*eme(ene[t-h])*mrho(t-h)*S(h ) \qquad
       end;

       calculo:=lc/p

end;

\medbreak

{\bf procedure eneinicial};

\{INPUT: random; \quad OUTPUT: first vector $N\in\R^{A_1p+1}$\}

 begin

     for i:=1 to A1*p do begin nhi[i]:=0; rnhi[i]:=0 end;

     for i:= 1 to A1*p do
     begin

          AZAR(l);
          nhi[i]:=l +500;
          \{ we assume that at least 500 rodents are alive\}

          rnhi[i]:=nhi[i]/55000  \{we normalize values; $N_t:=1$ means 55000 rodents\}

      end;

     for i:=1 to A1*p do rnhaux[i]:=rnhi[i];

     for i:=A1*p+1 to 2*A1*p+1 do rnhaux[i]:=0;

     rnhi[A1*p+1]:=calculo(A1*p+1,rnhaux);

\{warning: the coordinates of the vector $N$ begin with 1 and
finishes with A1*p+1\}

end;

\medbreak

{\bf procedure rnhgen(ene:rentrada;var ere:rentrada)};

\{INPUT: $N\in \R^{A_1p+1}$\quad OUTPUT: $T^2(N)\in\R^{A_1p+1}$\}

  type

    rentrada = array[1..A1*p+1] of extended;

 begin

      t:=1;bo:=A1*p+1;

      for j:= 1 to bo do begin  rnhaux[j]:=ene[j]; ere[j]:=0 end;

      for j:=bo+1 to 2*A1*p+1 do rnhaux[j]:=0;

     for t:=bo+1 to 2*A1*p+1 do begin z:=calculo(t,rnhaux);

     rnhaux[t]:=rnhaux[t]+z end;

      for i:=1 to bo do ere[i]:=rnhaux[i+A1*p]

 end;

\medbreak

begin {\bf\{of program ``ratones''\}}

\{INPUT: parameter values, random data \quad \par  OUTPUT: $na$
files of data representing time series of population of {\it Microtus
Epiroticus} \}

for jj:=1 to na do

  begin

   rewrite(datos[jj]);
   comienzoazar;

   writeln('generating datos[',jj,']');

   eneinicial;
   rnhgen(rnhi,rnh);

   for j:=1 to 10000 do
    begin

       rnhv:=rnh;
       rnhgen(rnhv,rnh)

    end;  \{20000 iterates of T: N-->T**(20000)(N)\}

   for i:=1 to pipa do
     begin

         archi[i].numero:=0;

         for j:=1 to A1*p+1 do
         archi[i].serie[j]:=0;

      end;

   archi[1].serie:=rnhi; archi[2].numero:=20000; archi[2].serie:=rnh;

   for i:=3 to pipa do
   begin

              rnhv:=rnh;
              rnhgen(rnhv,rnh); \{2 iterates of T each time\}

      archi[i].numero:=20000+2*(i-2); archi[i].serie:=rnh;

   end;

   for i:=1 to pipa do begin write(datos[jj],archi[i]); end;

 end; \{of ``for jj''\}

  writeln('type any key to finish'); ch:= readkey; exit

end. \{of ``ratones''\}

\begin{center}
----------------------------------------------------------
\end{center}

{\bf procedure AZAR1(n: extended)};

\{INPUT: random\_seed\quad OUTPUT: pseudo-random number$\in\R$\}

begin

   x:=random;
   n:=x-0.5;

 end; \{of AZAR1\}

\medbreak

{\bf procedure eneperturb1};

\{INPUT: $rnhi\in\R^{A_1p+1}$\quad OUTPUT: $rnhi+\Delta rnhi\in
\R^{A_1p+1}$\}

 begin

     for i:=1 to A1*p do begin nhi[i]:=0; end;

     for i:= 1 to A1*p do
     begin

          AZAR1(l);
          nhi[i]:=l ;

          rnhi[i]:=rnhi[i]+nhi[i]/(2**50)

      end;

     for i:=1 to A1*p do rnhaux[i]:=rnhi[i];

     for i:=A1*p+1 to 2*A1*p+1 do rnhaux[i]:=0;

     rnhi[A1*p+1]:=calculo(A1*p+1,rnhaux);

end;

\medbreak

begin {\bf \{of  program ``ratones1''\}}

\{INPUT: a file ''datos'' generated by ``ratones''\quad \par
OUTPUT:
a file ``datosp'' representing an initial small perturbation of
``datos''\}

for ii:=1 to na do

begin

  rewrite(datosp[ii]);
  comienzoazar;

 reset(datos[ii]); xx.numero:=-1;  ayuda:=true;

  while (not Eof(datos)) and (ayuda=true)  do

 begin

 read(datos,xx);

 if xx.numero=0 then begin archi[1].numero:=0; archi[1].serie:=xx.serie end;

 write(xx.numero,' serie ',xx.serie[1],' | ', xx.serie[100]);

 writeln;

 if xx.numero=20000 then begin

                             rnhi:= xx.serie; ayuda:=false

                          end;

  end;

  eneperturb1;
  rnh:=rnhi;

   for i:=2 to pipa do

     begin

         archi[i].numero:=0;

         for j:=1 to A1*p+1 do
         archi[i].serie[j]:=0;
      end;

   archi[2].numero:=20000; archi[2].serie:=rnh;

   for i:=3 to pipa do

   begin

         for j:=1 to 1 do

         begin

              rnhv:=rnh;
              rnhgen(rnhv,rnh);

           end;

      archi[i].numero:=20000+2*(i-2); archi[i].serie:=rnh;

   end;

   for i:=1 to pipa do begin write(datosp,archi[i]); end;

   reset(datos); reset(datosp);

   while (not Eof(datos)) and (not Eof(datosp)) do

   begin

   read(datos,xx); read(datosp,yy);

   write(xx.numero,' serie ',xx.serie[1],' | ',yy.numero,' serie ', yy.serie[1]);

   writeln;

   end
end;

  writeln('press any key to finish'); ch:= readkey; exit

end. \{of ``ratones1''\}

\medbreak

\begin{center}
--------------------------------------------------------------
\end{center}

\medbreak

{\bf function comparar(rnhx,rnhy: rentrada):longint;}

\{INPUT: $rnhx,rnhy\in\R^{A_1p+1}$\quad OUTPUT: 0 or 1\}

 \{if ``comparar'' =0 then $||rnhx-rnhy||<tol$, if  1 then $>$ 0 \}

 begin

  i:=1;

   cmaux:=0; \{we assume that at the beginning ``comparar'' is 0\}

   while (cmaux=0) and (i$<=$A1*p) do

   begin

        if (abs(rnhx[i]-rnhy[i])$>=$tol) then cmaux:=1;
        i:=i+1;

     end;

    comparar:=cmaux;

  end;

\medbreak

begin {\bf \{of program ``entropia3''\}}

\{INPUT: $na$ files generated by ``ratones'', \quad
\par OUTPUT: 16
files with pairs of time series $d_i$-near, $i=1,2,\ldots 16$; mean
value $<N>$ of $N_t$; \par absolute standard deviation of $N_t$\}

 for jj:=1 to na do
   begin

    reset(datos[jj]);
    j:=0; z[jj]:=0;

   while (not Eof(datos[jj]))  do
   begin

   for i:=1 to 8 do
    begin

    y:=0;
    read(datos[jj],xx);  j:=j+1;

    for h:=1 to A1*p do y:=y+xx.serie[h];

    y:=y/(A1*p);
    z[jj]:=z[jj]+y; \qquad
    end

   end;  \{of ``while not Eof''\}

   z[jj]:=z[jj]/j;

  writeln('the mean value of file datos[',jj,'] is:  ', z[jj]);

  end; \{of `` for jj''\}

  prom:=0;
  for jj:=1 to na do prom:= prom+z[jj];

  prom:=prom/na;
  writeln('total mean value ', prom);

  writeln('press any key to continue'); readkey(leer);

  for jj:=1 to na do
   begin

     reset(datos[jj]); j:=0; w[jj]:=0;

     while (not Eof(datos[jj])) do
     begin

     y:=0;

     read(datos[jj],xx); j:=j+1;

     for h:=1 to A1*p do y:=y+abs(xx.serie[h]-z[jj]);

     y:=y/(A1*p);
     w[jj]:=w[jj]+y;

     end; \{of ``while''\}

     w[jj]:=w[jj]/j;

     writeln('the absolute deviation value for datos[',jj,'] is ', w[jj]);

  end; \{of ``for jj''\}

  dis:=0;

  for jj:=1 to na do dis:=dis+w[jj];
  dis:=dis/na;

  writeln('total deviation = ', dis);

  writeln('to continue press ENTER'); readln(leer);

  \{we collect data\}

 tol:=1/(2**(10));

 \{``tol'' is what is called $d$ in the algorithm; here we exemplify with $tol\approx 0.001$ \}

  rewrite(info);

  for jj:=1 to na-1 do
  begin

   reset(datos[jj]);

   If (not Eof(datos[jj])) then read(datos[jj],xx); \{we discard the first\}

   for ii:=jj+1 to na do
    begin

      while (not Eof(datos[jj])) do
       begin

         read(datos[jj],xx);
          reset(datos[ii]);

        If (not Eof(datos[ii])) then read(datos[ii],yy); \{we discard the first\}

         while (not Eof(datos[ii])) do
          begin

            read(datos[ii],yy);

            u:=comparar(xx.serie,yy.serie);

            if u=0 then  \{that is: $||T^l(N)-T^l(N')||<tol$\}

               begin

                estx.numarch1:=jj; estx.numarch2:=ii;

                estx.numiter1:=xx.numero; estx.numiter2:=yy.numero;

                estx.punto1:=xx.serie; estx.punto2:=yy.serie;

                write(info,estx);

              end \{of ``if''\}

          end \{of ``while not Eof(datos[ii])''\}

       end \{of ``while not Eof(datos[jj])''\}

    end \{of ``for ii''\}
  end \{of ``for jj''\}

  writeln('teclee cualquier tecla para finalizar'); ch:= readkey;
  exit;

  end. \{of program ``entropia''\}

  \begin{center}
-----------------------------------------------------------------------
  \end{center}

\medbreak

  begin {\bf \{of program ``entropia4''\}}

\{INPUT: A file with pairs of time series $d_i$-near,\quad
\par OUTPUT: an estimation of the second order Kolmogorov-entropy
$\widetilde K$; \par an estimation of its standard deviation
$\sigma_{\widetilde K}$\}

   rewrite(androide);

   for jj:=1 to na do
    reset(datos[jj]);

  tol:=1/(2**(10))

   begin

    base:=1;
    reset(info);

    while (not Eof(info)) do
     begin

      read(info,estx);

      if base$>$tope then begin writeln('error, table too small'); halt end;

      tabla[base].numarch1:=estx.numarch1; tabla[base].numarch2:=estx.numarch2;

      tabla[base].numiter1:=estx.numiter1; tabla[base].numiter2:=estx.numiter2;

      rnh1:=estx.punto1; rnh2:=estx.punto2; j:=0;

      repeat

              rnh1v:=rnh1;
              rnhgen(rnh1v,rnh1);
              rnh2v:=rnh2;
              rnhgen(rnh2v,rnh2);
              j:=j+1;

      until comparar(rnh1,rnh2)$<>$0;

      tabentr[base]:=j;

      if base$>$1 then
        begin

         if (tabla[base].numarch2=tabla[base-1].numarch2) and

            (tabla[base-1].numiter2+tabentr[base-1]$>=$tabla[base].numiter2)

                       then base:=base-1 \{overlap of data\}

                       else

          begin

           if (tabla[base].numarch1=tabla[base-1].numarch1) and

             (tabla[base-1].numiter1+tabentr[base-1]$>=$tabla[base].numiter1)

                       then base:=base-1 \{overlap of data\}

          end
        end;

      base:=base+1;

     end; \{of ``while not Eof''\}

   tiempos:=0;

   for i:=1 to base-1 do
    begin

     tiempos:=tiempos+tabentr[i];
     end;

   tiempos:=tiempos/(base-1);
   entropy:=-Ln(abs(1-1/tiempos));

  writeln('the values of $b_j$ are');

  for j:=1 to base -1 do

  begin
    writeln('b',j,' = ',tabentr[j],' $ |$  ');
    end;

    writeln('average of $b_j$ is $ < $b $>$ = ',tiempos);

 writeln(' Entropy estimated is  ',entropy, ', the size of the sample is ',base-1);

 rna:=base-1;

 writeln('standard deviation of K is :', 1/(sqrt(rna)*entropy*sqrt(tiempos*(tiempos-1))));

 resultado:=entropy;

 desvio:= 1/(sqrt(rna)*entropy*sqrt(tiempos*(tiempos-1)));

 writeln('to finish press any key '); ch:=readkey;

end. \{of program ``entropia4''\}

}

\subsection{Appendix D: pseudo-code of homclin4.}
\footnotesize{  {\bf Program ``homclin4''}\\
\{INPUT: a table with the candidates to be homoclinic points\}\\
\{OUTPUT: A point in $W^u_{loc}(p)$ such that near it there is numerical evidence
that it exists a homoclinic point\}\\

This program uses, apart from the functions and procedures defined above,
 two functions ``distl2'' and ''angulo''.
 ''distl2'' computes the Euclidean distance between points, and ``angulo''
computes the angle between a pair of vectors. ``angulo`` uses a function
 ``prodint'' that calculates the inner product between vectors. The program
 also uses two procedures,
 ``minimo'' that computes the minimum
between real data stored in a file called ``candihomclin'' and
 ``iterar'' that iterates the function $T^2$ a prescribed number of times.

 \medbreak

  {\bf function distl2(rnhx,rnhy: rentrada):extended;}

 \{calculates  euclidean distance between points\}

 var cmaux,i:longint; raux,dist:extended; maximo:extended;
     rnhd:rentrada;

 begin

  i:=2; maximo:=abs(rnhx[1]-rnhy[1]);

  while i$ <=$A1*p+1 do
  begin

      if  abs(rnhx[i]-rnhy[i])$ > $ maximo then maximo:=abs(rnhx[i]-rnhy[i]);

      i:=i+1 \qquad
  end;

  if maximo$<>$0 then

    for i:=1 to A1*p+1 do rnhd[i]:= abs(rnhx[i]-rnhy[i])/maximo;

  i:=1;
   dist:=0; \{ assume distance is 0\}

   if maximo$<>$0 then

   while (i$<=$A1*p+1) do
   begin

        dist:=dist+rnhd[i]*rnhd[i];
        i:=i+1;\qquad
     end;

    distl2:=maximo*sqrt(dist);

  end;

\medbreak

  {\bf function prodint(rnhx,rnhy: rentrada):extended;}

  \{computes inner product of vectors\}

  var i:longint; prod:extended; rnhd,rnhe:rentrada;maximox,maximoy:extended;

  begin

       i:=2; maximox:=abs(rnhx[1]);maximoy:=abs(rnhy[1]);

       while i$<=$A1*p+1 do
       begin

        if  abs(rnhx[i]) $>$ maximox then maximox:=abs(rnhx[i]);

        if  abs(rnhy[i]) $>$ maximoy then maximoy:=abs(rnhy[i]);

        i:=i+1 \qquad
       end;

  if (maximox*maximoy$<>$0)
  then
    begin

      for i:=1 to A1*p+1 do

          begin rnhd[i]:= rnhx[i]/maximox; rnhe[i]:=rnhy[i]/maximoy end;

     i:=1; prod:=0;

     while i$<=$A1*p+1 do
       begin

           prod:=prod+rnhd[i]*rnhe[i];
           i:=i+1 \qquad
       end;

     prodint:=prod*maximox*maximoy; \qquad
    end

   else prodint:=0;

  end;

\medbreak

 {\bf function angulo(rnhx,rnhy:rentrada):extended;}

      var equis, ye, zeta:extended;

   begin

       zeta:= prodint(rnhx,rnhy);
       equis:=sqrt(prodint(rnhx,rnhx));
       ye:=sqrt(prodint(rnhy,rnhy));

       if (equis=0) or (ye=0) then angulo:=0
       else angulo:=arccos(zeta/(equis*ye));

   end;

\medbreak

{\bf  procedure minimo;}

 var minaux:extended;fijmin:rentrada; seguir:boolean;

  begin

   min:=1; seguir:=true; \{``min'' is set to a value which will not be the minimum\}

   while (not Eof(refcandihomclin)) and (seguir=true) do
    begin

       read(refcandihomclin,homocl);
       if homocl.punto[2]$<>$0 then
         begin

          seguir:=false; min:=10**(-2);
          refhomocl:=homocl \qquad
         end;

       if (homocl.punto[2]=0) and (homocl.punto[1]$<$min) and (homocl.numh$<$700)

          and (homocl.numh$>$10)
        then
         begin

            min:=homocl.punto[1]; refhomocl:=homocl \qquad
         end;

     end; \{of while\}

  end;

\medbreak

{\bf  procedure iterar(paso:integer;sota:rentrada;var sota1:rentrada);}

  var rnhj,rnhjv:rentrada; \{``paso'' controls the number of iterations\}

  begin rnhjv:=sota;

        for j:=1 to paso do

     begin
      rnhgen(rnhjv,rnhj);
      rnhjv:=rnhj; \qquad
     end;

   sota1:=rnhj;

  end;

\medbreak

{\bf begin }\{of homclin4\}

  while not Eof(candihomclin) do
   begin

        read(candihomclin,homocl);

        if (homocl.punto[2]=0.0) and (homocl.punto[1]<0.001) then

                                   write(refcandihomclin,homocl);

        if homocl.punto[2]<>0.0 then write(refcandihomclin,homocl)

   end; \{of while\}

  reset(refcandihomclin);

  minimo;

  writeln('minimum distance to p is ', min);

  writeln('value of i=',refhomocl.numi,'  iterate closest to p is ',refhomocl.numh+10);

  writeln('initial approximation to candidate to homoclinic point M is  ');

   for j:=1 to A1*p+1 do
   begin

    fijo12[j]:=fijo6[j]+(10000-refhomocl.numi)*fijo8[j];

    if (j mod 3=0)  then writeln(fijo12[j],'|')

                  else write(fijo12[j],' |'); \qquad
   end;

 tolerancia: \{a label of reference\}

  if (refhomocl.numh mod 2 = 0) then techo:=refhomocl.numh+10

                                else techo:=refhomocl.numh+10;

  fijo8:=restar(fijo6,fijo4);

  for j:=1 to A1*p+1 do fijo8[j]:=fijo8[j]/10000;

 writeln;

 writeln('Next we refine the choice, in particular we find L and R');

 writeln('points in $[T^{38}(p),T^{40}(p)]$ identified with $W^u_{loc}(p)$');

 writeln('such that M is between them and such that the iterates');

 writeln('of L and R are in different components with respect to');

 writeln('the local stable manifold $W^s_{loc}(p)$ of $p$.');

 writeln('For convenience we continue to denote by M, L and R their iterates by $T^2$');

 writeln('Enter gap distance as a real exponent of 2 not greater than 30');

 writeln('the gap distance will be $2^{(-exponent)}$');

  write('To finish the program enter exponent=0, exponente = ');

 readln(semillon);

 if semillon$<0$ then

 begin
  semillon:=-semillon;

  writeln('a negative value has been entered, ',semillon,' will be assumed');

 end;

  if semillon$>20$ then

  begin

    writeln('exponent too large, a value of 10 will be assumed');

    semillon:=10
  end;

 while semillon$<>0$ do

 begin

  tol:=2**(semillon);  writeln('tol=',1/tol);

   for j:=1 to A1*p+1 do

   begin

    fijo12[j]:=fijo6[j]+(10000-refhomocl.numi)*fijo8[j];

    fijo11[j]:=fijo12[j]-fijo8[j]/tol;

    fijo13[j]:=fijo12[j]+fijo8[j]/tol;

    fijo115[j]:=fijo12[j]-fijo8[j]/(2*tol);

    fijo135[j]:=fijo12[j]+fijo8[j]/(2*tol);

   end;

  while techo$>0$ do

  begin

   if techo$>= 10$ then
          begin

             iterar(10,fijo12,rnh1);
             iterar(10,fijo11,rnh0);

             iterar(10,fijo13,rnh2);
             iterar(10,fijo115,rnh05);

             iterar(10,fijo135,rnh25);
             techo:=techo-10;

             writeln('distance between left and right iterates L and R is ',distl2(rnh0,rnh2));

             writeln('dist(L,L1)+dist(L1,M)+dist(M,R1)+dist(R1,R)= ',

             distl2(rnh0,rnh05)+distl2(rnh05,rnh1)+distl2(rnh1,rnh25)+distl2(rnh25,rnh2));

             if (distl2(rnh0,rnh2) $>$ 0.0001) or

             (distl2(rnh0,rnh05)+distl2(rnh05,rnh1)+distl2(rnh1,rnh25)+distl2(rnh25,rnh2)

             $ > $ 1.001*distl2(rnh0,rnh2))

              then
               begin

                writeln('distance between iterates is too large or curvature is big');

                techo:=techo+10;
                iterar(8,fijo12,rnh1);

                iterar(8,fijo11,rnh0);
                iterar(8,fijo13,rnh2);

                iterar(8,fijo115,rnh05);
                iterar(8,fijo135,rnh25);

                techo:=techo-8;

                writeln('iterating 8 times the new distance between L and R is ',distl2(rnh0,rnh2));

                writeln('dist(L,L1)+dist(L1,M)+dist(M,R1)+dist(R1,R)= ',

             distl2(rnh0,rnh05)+distl2(rnh05,rnh1)+distl2(rnh1,rnh25)+distl2(rnh25,rnh2));

             if (distl2(rnh0,rnh2) $>$ 0.0001) or

             (distl2(rnh0,rnh05)+distl2(rnh05,rnh1)+distl2(rnh1,rnh25)+distl2(rnh25,rnh2)

             $ > $ 1.001*distl2(rnh0,rnh2))

              then
               begin

                 writeln('distance between iterates continues to be too large or curvature is big');

                 techo:=techo+8;
                 iterar(2,fijo12,rnh1);

                 iterar(2,fijo11,rnh0);
                 iterar(2,fijo13,rnh2);

                 iterar(2,fijo115,rnh05);
                 iterar(2,fijo135,rnh25);

                 techo:=techo-2;

                 writeln('iterating 2 times the new distance between L and R is ',distl2(rnh0,rnh2));

                 writeln('dist(L,L1)+dist(L1,M)+dist(M,R1)+dist(R1,R)= ',

                 distl2(rnh0,rnh05)+distl2(rnh05,rnh1)+distl2(rnh1,rnh25)+distl2(rnh25,rnh2));

                          end \{of inner ``if then''\}

                        end
                      end \{of outer ``if then''\}

                 else begin \{ now ``techo'' is less or equal than 10\}

                      iterar(techo,fijo12,rnh1);
                      iterar(techo,fijo11,rnh0);

                      iterar(techo,fijo13,rnh2);
                      techo:=0

                      end;

     fijo8:=restar(rnh2,rnh0);
     fijo12:=rnh1;

      for j:=1 to A1*p+1 do

    begin

     fijo11[j]:=fijo12[j]-fijo8[j]/tol;
     fijo13[j]:=fijo12[j]+fijo8[j]/tol;

     fijo115[j]:=fijo12[j]-fijo8[j]/(2*tol);
     fijo135[j]:=fijo12[j]+fijo8[j]/(2*tol);

    end;

    writeln(' iterates= ',refhomocl.numh+10-techo,'

    distance between endpoints L and R previous to iteration is ');

    writeln(distl2(fijo11,fijo13));

   {if distl2(fijo11,fijo13) > 0.00001 then
    for j:=1 to A1*p+1 do

     begin

      fijo11[j]:=fijo12[j]-fijo8[j]/(16*tol);
      fijo13[j]:=fijo12[j]+fijo8[j]/(16*tol)

     end;}

   if distl2(fijo11,fijo13)<=0.0000000000000001 then begin

   writeln('tol is too small, please, reduce the exponent'); goto tolerancia; end;

  end; \{del while techo\}

  dist0:=distl2(fijo,rnh0);
  dist1:=distl2(fijo,rnh1);

  dist2:=distl2(fijo,rnh2);
  writeln;

  writeln(' Euclidean dist from p to original point $T^2$',*(refhomocl.numh+10),'is ',dist1);

  writeln(' Euclidean dist from p to left point ',dist0);

  writeln(' Euclidean dist from p to right point ',dist2);

  writeln(' Euclidean dist between left and right points is ');
  writeln(distl2(rnh0,rnh2));

   if (refhomocl.numh mod 2 = 0) then techo:=refhomocl.numh+10

                                else techo:=refhomocl.numh+10;

   fijo14:=restar(fijo4,fijo2);
   fijo16:=restar(rnh2,rnh0);
   rnhgen(fijo6,fijo8);

   writeln('angle between $W^u_e(p)$ and iterated arc LM is =  ');

   write(angulo(fijo14,fijo16));

  writeln(' angle in degrees is approx = ' ,round(angulo(fijo14,fijo16)*180/Pi));

  writeln(' Euclidean dist between left end-point of $W^u_e(p)$ and L is ');

  writeln(distl2(fijo6,rnh0));

   writeln(' Euclidean dist between left end-point of $W^u_e(p)$ and R is ');

   writeln( distl2(fijo6,rnh2));

   writeln(' Euclidean dist between right end-point of $W^u_e(p)$ and L is ');

   writeln(distl2(fijo8,rnh0));

   writeln(' Euclidean dist between right end-point of $W^u_e(p)$ and R is ');

   writeln(distl2(fijo8,rnh2));

   rnhgen(rnh0,rnh0v);rnhgen(rnh2,rnh2v);

  { writeln('rate of dist between rnh0, rnh2 and their iterates by $T^2$ is ');

   writeln(distl2(rnh0v,rnh2v)/distl2(rnh0,rnh2));}

   fijo18:=restar(fijo,rnh0);
   fijo20:=restar(fijo,rnh2);

   writeln('angle between vectors (p,L) and (p,R) is ');

   write(angulo(fijo18,fijo20));

   writeln(' angle in degrees is approx = ',round(angulo(fijo18,fijo20)*180/Pi));

    fijo18:=restar(fijo,rnh0v);
   fijo20:=restar(fijo,rnh2v);

   writeln('angle between vectors (p,$T^2(L)$) and (p,$T^2(R)$) is ');

   write(angulo(fijo18,fijo20));

   writeln(' angle in degrees is approx = ',round(angulo(fijo18,fijo20)*180/Pi));

   for ii:=1 to 6 do
   begin

    newfix[2*ii-1]:=rnh0v;newfix[2*ii]:=rnh2v;

    rnhgen(newfix[2*ii-1],rnh0v);rnhgen(newfix[2*ii],rnh2v);

    fijo18:=restar(fijo,rnh0v);
    fijo20:=restar(fijo,rnh2v);

    writeln('angle between vectors (p,$T^2$',*(ii+1),'(L)) and (p,$T^2$',*(ii+1),'(R)) is ');

    write(angulo(fijo18,fijo20));

    writeln(' angle in degrees is approx = ',round(angulo(fijo18,fijo20)*180/Pi));

   end;
   writeln;

   nuevofijo2:=restar(nuevofijo,rnh0);
   nuevofijo4:=restar(nuevofijo,rnh2);

   writeln('angle between vectors $(\hat p,L)$ and $(\hat p,R)$ is ');

   write(angulo(nuevofijo2,nuevofijo4));

   writeln(' angle in degrees is approx = ',round(angulo(nuevofijo2,nuevofijo4)*180/Pi));

   writeln;
 fijo8:=restar(fijo6,fijo4);

    for j:=1 to A1*p+1 do fijo8[j]:=fijo8[j]/10000;

    writeln;
   writeln('Enter gap distance as a real exponent of 2 no greater than 30');

   writeln('last exponent used is ',semillon);

  write('To finish the program enter exponent=0, exponent = ');

 readln(semillon);

 if semillon$<0$ then
     begin

      semillon:=-semillon;

      writeln('a negative value has been entered, ',semillon,' will be assumed');

     end;

  if semillon$>20$ then
  begin

    writeln('exponent too large, a value of 10 will be assumed');

    semillon:=10;

  end
 end;

     {  write(candihomclin2,homocl);}

  writeln('To continue press ENTER'); read(leer);
  writeln('Press any key to finish'); ch:=readkey;

end.

}
\subsection{Appendix E: coordinates of fixed point.}
Approximate coordinates of the fixed point $p\in\R^{201}$ of $T^2(N)$ are given in the following table.
{\footnotesize
$$
\begin{array}{ccc}
 1.2326490487970465E+0000 & 1.2110482116814741E+0000 & 1.1906886685005045E+0000 \\
 1.1717713776064593E+0000 & 1.1545319083055524E+0000 & 1.1392463406313234E+0000 \\
 1.1262379872047533E+0000 & 1.1158849347082125E+0000 & 1.1086283238645345E+0000 \\
1.1049811591622351E+0000 & 1.1055372415914759E+0000 & 1.1109795220165019E+0000 \\
1.1220867468556205E+0000 & 1.1397366769671205E+0000 & 1.1649033774342367E+0000 \\
1.1986450973703732E+0000 & 1.2420781384673748E+0000 & 1.2087810376155805E+0000 \\
 1.1754839367637863E+0000 & 1.1421868359119921E+0000 & 1.1088897350601978E+0000\\
1.0755926342084036E+0000 & 1.0422955333566094E+0000 & 1.0089984325048151E+0000 \\
 9.7570133165302088E-0001 & 9.4240423080122665E-0001 & 9.0910712994943241E-0001\\
8.7581002909763817E-0001 & 8.4251292824584393E-0001 & 8.0921582739404970E-0001 \\
 7.7591872654225546E-0001 & 7.4463460038547187E-0001 & 7.1528062031461468E-0001\\
6.8777896512033852E-0001 & 6.6205661495172081E-0001 & 6.3804515773116380E-0001\\
 6.1568060751297437E-0001 & 5.9490323430406656E-0001 & 5.7565740489494235E-0001\\
 5.5789143427761938E-0001 & 5.4155744725456250E-0001 & 5.2661124986901219E-0001 \\
  5.1301221031245340E-0001 & 5.0072314898939969E-0001 & 4.8971023744324276E-0001 \\
4.7994290586969646E-0001 & 4.7139375896640730E-0001 & 8.3241286907550863E-0001 \\
1.1774339903181779E+0000 & 1.5083543629435397E+0000 & 1.8248209517576603E+0000 \\
2.1271151896510347E+0000 & 2.4160208699409385E+0000 & 2.6923595086534992E+0000 \\
 2.9569342453691580E+0000 & 3.2105209755162044E+0000 & 3.4538676455028848E+0000 \\
3.6876953663770639E+0000 & 3.9126999878684479E+0000 & 4.1295537791098290E+0000 \\
 4.3389071105434127E+0000 & 4.5413901032229220E+0000 & 4.7376142351332896E+0000 \\
4.9281739025629719E+0000 & 5.1136479378130289E+0000 & 5.4622938956939894E+0000 \\
5.5692297788707215E+0000 & 5.5592343648444013E+0000 &  5.5318006317643429E+0000 \\
5.5004355530098589E+0000 & 5.4679067992365053E+0000 & 5.4349605692581039E+0000 \\
 5.4018415509009528E+0000 & 5.3686426726693849E+0000 &  5.3354035019177187E+0000 \\
5.3021425008934430E+0000 & 5.2688689546762325E+0000 & 5.2355878340449946E+0000 \\
5.2023019446101286E+0000 & 5.1690129433312447E+0000 &  5.1357218479879554E+0000 \\
5.1024293056141891E+0000 & 5.0691357402363859E+0000 & 5.0358409167684584E+0000 \\
5.0025458322320211E+0000 & 4.9692507815459950E+0000 &  4.9359558081745136E+0000 \\
4.9026609256467161E+0000 & 4.8693661415807602E+0000 & 4.8360714621823154E+0000 \\
 4.8027768933916233E+0000 & 4.7694824412520072E+0000 &  4.7361881120565536E+0000 \\
4.7028939124217354E+0000 & 4.6695998493345792E+0000 & 4.6363059301899531E+0000 \\
 4.6030121628245493E+0000 & 4.5697185555504526E+0000 &  4.5364251171897114E+0000 \\
4.5031318571107021E+0000 & 4.4698387852668010E+0000 & 4.4365458667634358E+0000 \\
\end{array}
$$

$$
\begin{array}{ccc}
 4.4032531103361234E+0000 & 4.3699605252784317E+0000 &  4.3366681214830976E+0000 \\
4.3033759094851984E+0000 & 4.2700839005086080E+0000 & 4.2367921065161327E+0000 \\
 4.2035005402636639E+0000 & 4.1702092153587075E+0000 &  4.1369181463236780E+0000 \\
4.1036273486643819E+0000 & 4.0703368389441578E+0000 & 4.0370466348641843E+0000 \\
 4.0037567553505172E+0000 & 3.9704672206484750E+0000 &  3.9371780524250502E+0000 \\
3.9038809515732566E+0000 & 3.8705838507214630E+0000 & 3.8372867498696694E+0000 \\
 3.8039896490178758E+0000 & 3.7706925481660822E+0000 &  3.7373954473142887E+0000 \\
3.7040983464624951E+0000 & 3.6708012456107015E+0000 & 3.6375041447589079E+0000 \\
 3.6042070439071143E+0000 & 3.5709099430553207E+0000 &  3.5376128422035272E+0000 \\
3.5043157413517336E+0000 & 3.4710186404999400E+0000 & 3.4377218198045367E+0000 \\
 3.4044252993439741E+0000 & 3.3711291008502547E+0000 &  3.3378332478631211E+0000 \\
3.3045377659003338E+0000 &  3.2712426826458989E+0000 & 3.2379480281583430E+0000 \\
 3.2046538351014030E+0000 & 3.1713601389998044E+0000 &  3.1380669785231534E+0000 \\
3.1047743958013707E+0000 &  3.0714824367755503E+0000 & 3.0381911515886585E+0000 \\
 3.0049005950210884E+0000 & 2.9716108269767829E+0000 &  2.9383219130264393E+0000 \\
2.9050851936346499E+0000 &  2.8718528731824328E+0000 & 2.8386253205708071E+0000 \\
 2.8054029398018444E+0000 & 2.7721861737242195E+0000 &  2.7389755082227071E+0000 \\
2.7057714769095989E+0000 &  2.6725746663842937E+0000 & 2.6393857221368817E+0000 \\
 2.6062053551826095E+0000 & 2.5730343495269334E+0000 &  2.5398735705757476E+0000 \\
2.5067239746226584E+0000 &  2.4735866195652930E+0000 & 2.4404626770260829E+0000 \\
 2.4073534460803303E+0000 & 2.3742603688263723E+0000 &  2.3411850480701324E+0000 \\
2.3081292418906175E+0000 &  2.2750949308401385E+0000 & 2.2420843237904744E+0000 \\
 2.2090998874484834E+0000 & 2.1761443801015639E+0000 &  2.1432208902575048E+0000 \\
2.1103328809563707E+0000 &  2.0774842406657146E+0000 & 2.0446793418285290E+0000 \\
 2.0119231083205840E+0000 & 1.9792210932957181E+0000 &  1.9465795691608410E+0000 \\
1.9140056317346885E+0000 &  1.8815073210149785E+0000 & 1.8490937614183503E+0000 \\
 1.8167753248789589E+0000 & 1.7845638208093313E+0000 &  1.7524727176575964E+0000 \\
1.7205174017417589E+0000 &  1.6887154797718421E+0000 & 1.6570871330214122E+0000 \\
 1.6256555322735251E+0000 & 1.5944473242987124E+0000 &  1.5634932024257111E+0000 \\
1.5328285757821936E+0000 &  1.5024943539853568E+0000 & 1.4725378663866274E+0000 \\
1.4430139373005373E+0000 & 1.4139861407724177E+0000 &  1.3855282600450476E+0000 \\
1.3577259774985657E+0000 &  1.3306788197951572E+0000 & 1.3045023793744458E+0000 \\
1.2793308262228334E+0000 & 1.2553197117722144E+0000 &  1.2326490487971048E+0000 \\
\end{array}
$$
}
From the analytic expression of $T$, it is clear that $T(p)\neq p$ so that $p$ has period 2.

\subsection{Appendix F: homoclinic points search.}

We plot a projection of the attractor in three-dimensional space
(averaging some coordinates at the beginning of the year, in the middle of the year
and in Spring). For that purpose we use $\mbox{MATLAB}^{\copyright}$ . When we plot the image of the first 1000 iterates of a single point of the local unstable manifold $W^u_{loc}(p)$ we
roughly recover the image of $\Lambda$ obtained plotting all the sequences of points
pseudo-randomly generated. This is an indication that $W^u(p)$ may be dense in $\Lambda$.
The small red circle in the figures, indicates the approximate position
of the fixed point $p$.

\begin{figure}[ht]
%\begin{center}
\includegraphics[scale=0.45]{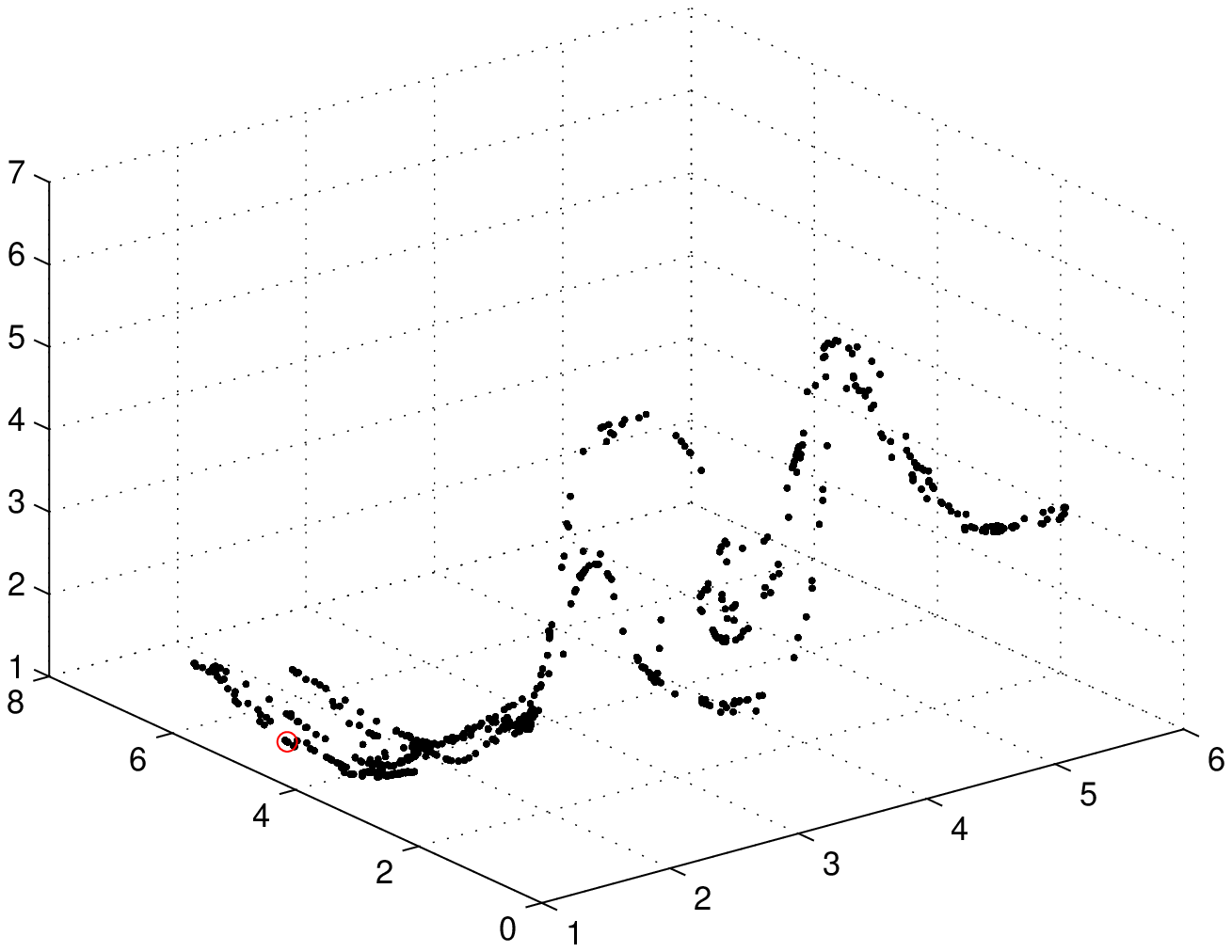}
\includegraphics[scale=0.45]{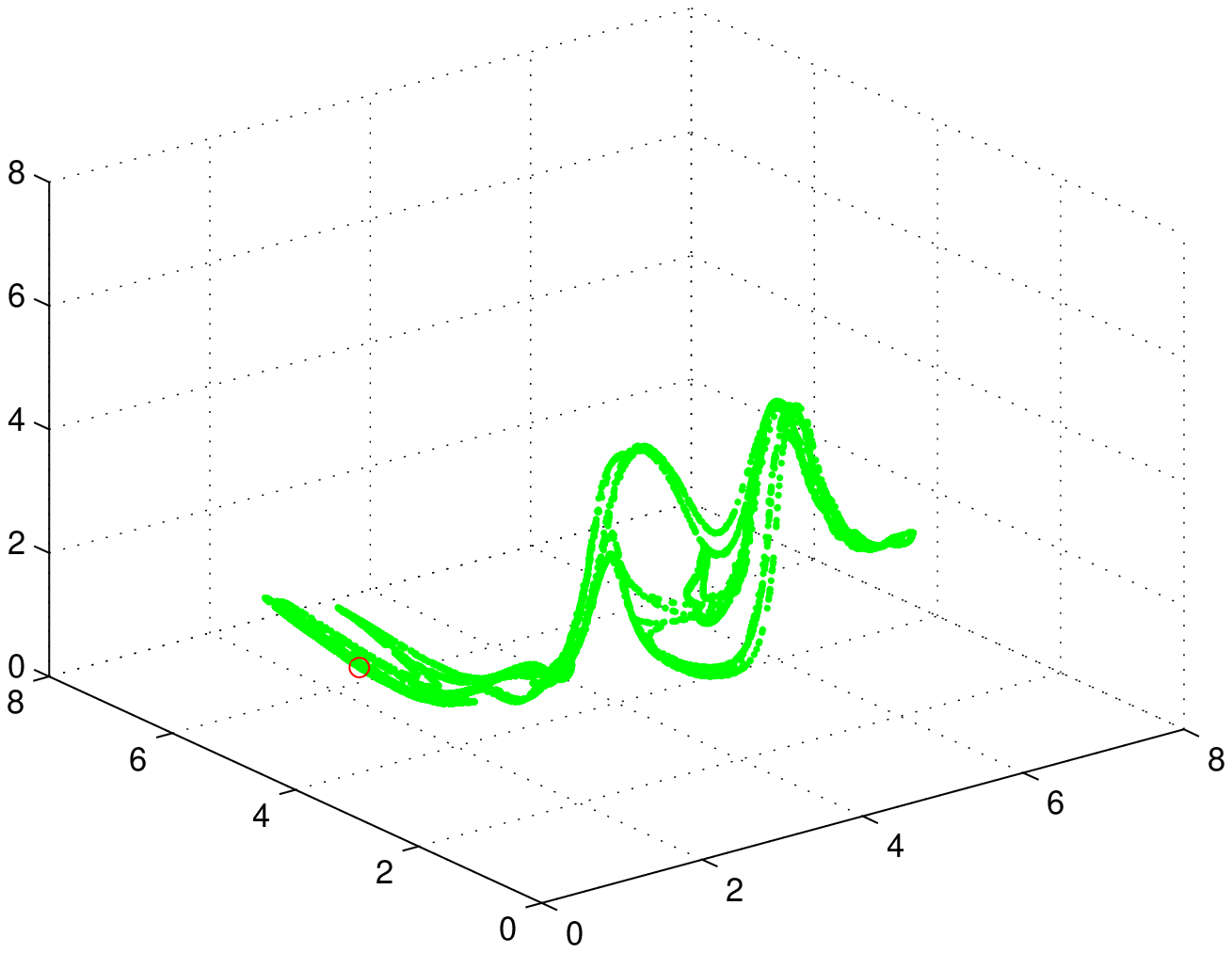}
\caption{A thousand iterates of $y\in W^u_{loc}(p)$ projected in $3D$ (in black) versus the
 projection of $\Lambda$ (in green).}\label{figatra3}
%\end{center}
\end{figure}

\begin{figure}[ht] \label{esquema2}
\begin{center}
\includegraphics[scale=0.46]{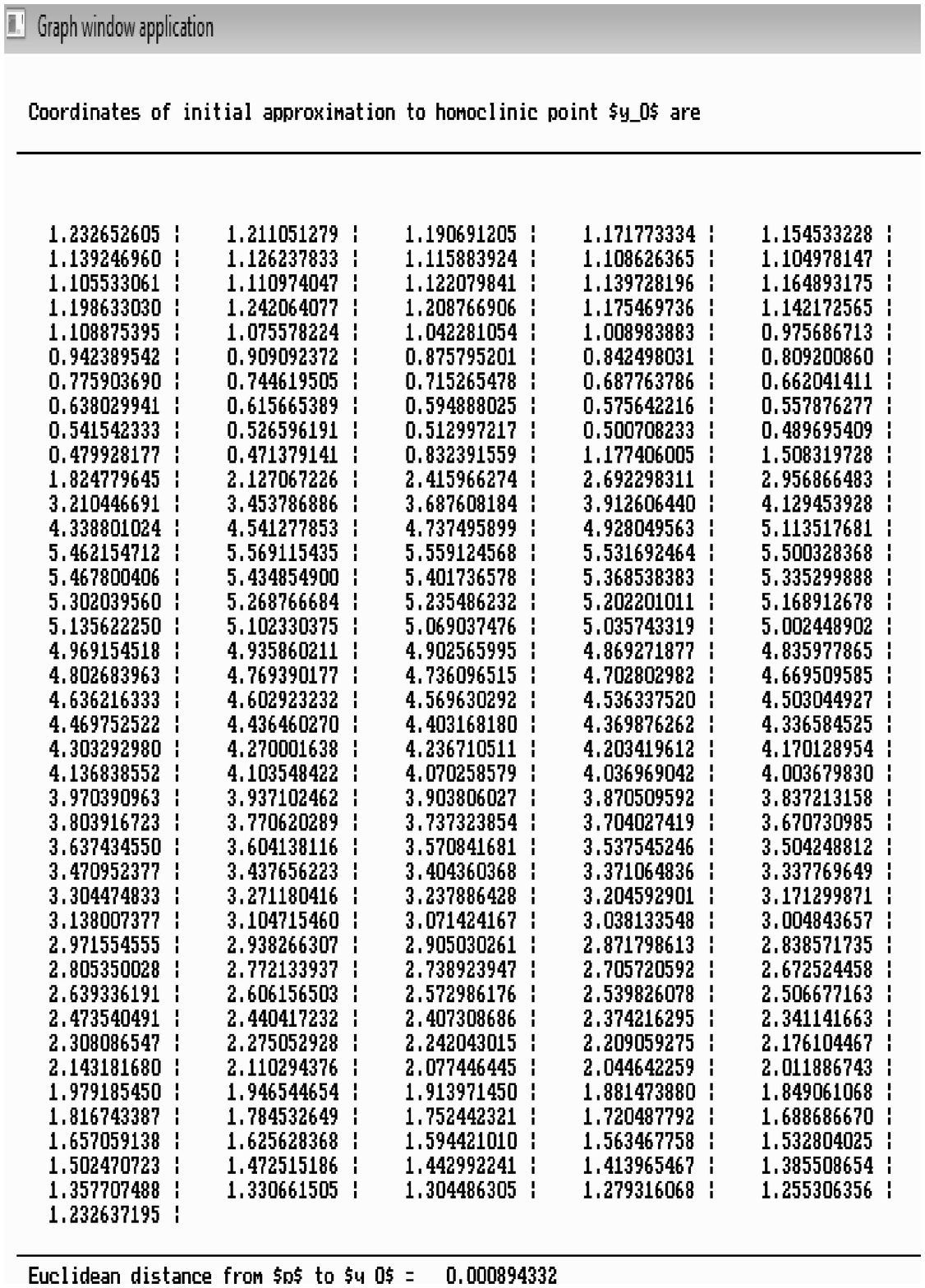}
\caption{Coordinates of (rough) homoclinic point $y_0$ associated to $p$.}
 \label{fig4}
\end{center}
\end{figure}

We also give approximate coordinates of the homoclinic point $y_0$ and
angular values obtained
with an initial length of $I_0$ of $ 1.122\times 10^{-7}$ in
the appendix below with two copies of runnings of ``homclin7.exe'' which is a refinement of ''homclin4.exe''
which generates an output close to LaTeX. In these runs we use three values for the parameter ''exponent'',
one of them is $10.80$ and the other is $15.03$. We also exhibit one exponent, $11.00$ which fails to detect homoclinic points. Observe that $10.80$ is not very far apart from $11.00$.

We only exhibit samples of the runs since they are rather extensive. It is possible to observe that the program corrects the quantity of iterations when the results are larger than certain bounds.

\subsection*{Runs of ''homoclin7''}
\footnotesize{
Enter gap as an exponent of 2 not greater than 20 and greater than 3, we choose gap=$2^{exponent} $. \newline                  This gap will be used to divide the distance between three consecutive points \newline                                     of the initial subdivision of $[T^{38}(p),T^{40}(p)]$ centered around \newline  the                                         rough homoclinic point $y_0$ previously found.           \\                   To finish the program enter exponent=0,                         \newline
\noindent
exponent =  1.0800000000000000E+0001                                        gap = 1782.8875536                   exponent = 10.8000000                                          \newline      iter= 0,                              dist(L,R) previous to iteration is  3.5890677421214318E-0010    \newline               dist between L and R after applying $T^{20}$ is  6.3294602762248157E-0007                                 \newline               dist(L,L1)+dist(L1,M)+dist(M,R1)+dist(R1,R)=  6.3294602764511862E-0007                                \newline               iter= 10, dist(L,R) previous to iteration is  7.1002349673070908E-0010                                \newline               dist between L and R after applying $T^{20}$ is  1.0395709474479636E-0008                                 \newline               dist(L,L1)+dist(L1,M)+dist(M,R1)+dist(R1,R)=  1.0395709474479637E-0008                                \newline               iter= 20, dist(L,R) previous to iteration is  1.1661654789708924E-0011                                \newline               dist between L and R after applying $T^{20}$ is  6.7123467133053106E-0010                                 \newline               dist(L,L1)+dist(L1,M)+dist(M,R1)+dist(R1,R)=  6.7123467133053267E-0010                                \newline
$\vdots \hspace{1cm} \vdots\hspace{1cm}  \vdots\hspace{1cm}\vdots\hspace{1cm}\vdots \hspace{1cm} \vdots\hspace{1cm}  \vdots\hspace{1cm}\vdots\hspace{1cm}\vdots$

\noindent
iter= 600, dist(L,R) previous to iteration is  1.7874017013569691E-0011                               \newline               dist between L and R after applying $T^{20}$ is  3.5203152021803773E-0006                                 \newline               dist(L,L1)+dist(L1,M)+dist(M,R1)+dist(R1,R)=  3.5203152021804383E-0006                                \newline               iter= 610, dist(L,R) previous to iteration is  3.9490041813784173E-0009                               \newline               dist between L and R after applying $T^{20}$ is  4.3191750633612666E-0005                                 \newline               dist(L,L1)+dist(L1,M)+dist(M,R1)+dist(R1,R)=  4.3191750633619148E-0005                                \newline               iter= 620, dist(L,R) previous to iteration is  4.8451457912339789E-0008                               \newline              dist between L and R after applying $T^{18}$ is  2.7846424585034373E-0003                                 \newline               dist(L,L1)+dist(L1,M)+dist(M,R1)+dist(R1,R)=  2.7846424590282723E-0003                                \newline               iter= 629, dist(L,R) previous to iteration is  3.1237443470096728E-0006                               \newline
\medbreak

\noindent
Sup distance from  fixed point p to point $T^{1258}(y_0)$ is  0.00022627982                           \newline            L1 distance from  fixed point p to point $T^{1258}(y_0)$ is  0.01670265394                            \newline            Euclidean distance from  fixed point p to point $T^{1258}(y_0)$ is  0.00141044319                     \newline            Euclidean distance from  fixed point p to point L is  0.00007862956                                   \newline            Euclidean distance from  fixed point p to point R is  0.00280177267                                   \newline            Euclidean distance between L and R is  0.00278464246           \newline            angle between $W^u_e(p)$ and iterated arc $LR$ =  0.00003 radians,         angle in degrees is $\approx$ 0                                \newline            angle between vectors $(p,L)$ and $(p,R)$ is  1.33746                      angle in degrees is $\approx$ 77                              \newline            angle between vectors $(p,T^2(L))$ and $(p,T^2(R))$ is  3.01126 radians,                                         angle in degrees is $\approx$ 173                              \newline            angle between vectors $(p,T^{4}(L))$ and $(p,T^{4} (R))$ is  3.13900 radians                                      angle in degrees is $\approx$  180                             \newline            angle between vectors $(p,T^{6}(L))$ and $(p,T^{6} (R))$ is  3.14031 radians                                      angle in degrees is $\approx$  180                             \newline            angle between vectors $(p,T^{8}(L))$ and $(p,T^{8} (R))$ is  3.13743 radians                                      angle in degrees is $\approx$  180                             \newline            angle between vectors $(p,T^{10}(L))$ and $(p,T^{10} (R))$ is  3.12545 radians                                    angle in degrees is $\approx$  179                             \newline            angle between vectors $(p,T^{12}(L))$ and $(p,T^{12} (R))$ is  3.10750 radians                                    angle in degrees is $\approx$  178                             \newline            angle between vectors $(p,T^{14}(L))$ and $(p,T^{14} (R))$ is  3.05523 radians                                    angle in degrees is $\approx$  175                             \newline
\medbreak

\noindent
Enter gap distance as a real exponent of 2 between 3 and 20   \newline            last exponent used is 10.80000000                              \newline               To finish the program enter exponent=0. Chosen exponent = 15.03200000   \newline                                         gap = 33502.9380910                  exponent = 15.0320000    \newline            iter= 0,                              dist(L,R) previous to iteration is  1.9099531465388560E-0011    \newline               dist between L and R after applying $T^{30}$ is  1.4460501331114946E-0008                                 \newline               dist(L,L1)+dist(L1,M)+dist(M,R1)+dist(R1,R)=  1.4460501331114946E-0008                                \newline               iter= 15, dist(L,R) previous to iteration is  8.6323753141429872E-0013                                \newline               dist between L and R after applying $T^{30}$ is  1.6945248777241631E-0009                                 \newline               dist(L,L1)+dist(L1,M)+dist(M,R1)+dist(R1,R)=  1.6945248777241637E-0009                                \newline               iter= 30, dist(L,R) previous to iteration is  1.0115684115597356E-0013                                \newline               dist between L and R after applying $T^{30}$ is  2.4222230946916936E-0009                                 \newline               dist(L,L1)+dist(L1,M)+dist(M,R1)+dist(R1,R)=  2.4222230946916942E-0009                                \newline
$\vdots\hspace{1cm}\vdots\hspace{1cm}\vdots\hspace{1cm}\vdots\hspace{1cm}\vdots\hspace{1cm}\vdots \hspace{1cm}\vdots\hspace{1cm}\vdots\hspace{1cm}\vdots\hspace{1cm}\vdots\hspace{1cm}\vdots$

\noindent
iter= 594, dist(L,R) previous to iteration is  1.2923688152099415E-0009                               \newline               dist between L and R after applying $T^{30}$ is  9.5328543032789122E-0004                                 \newline               dist(L,L1)+dist(L1,M)+dist(M,R1)+dist(R1,R)=  9.5328543041795548E-0004                                \newline            distance between iterates is too large or curvature is big     \newline               dist between L and R after applying $T^{16}$ is  6.3306020056154999E-0008                                \newline               dist(L,L1)+dist(L1,M)+dist(M,R1)+dist(R1,R)=  6.3306020056154999E-0008                                \newline               iter= 602, dist(L,R) previous to iteration is  3.7791322627648531E-0012                               \newline               dist between L and R after applying $T^{30}$ is  1.9472974658090783E-0005                                 \newline               dist(L,L1)+dist(L1,M)+dist(M,R1)+dist(R1,R)=  1.9472974658261263E-0005                                \newline               iter= 617, dist(L,R) previous to iteration is  1.1624636981219546E-0009                               \newline              dist between L and R after applying $T^{24}$) is  2.7688866606693677E-0003                                 \newline               dist(L,L1)+dist(L1,M)+dist(M,R1)+dist(R1,R)=  2.7688866611853443E-0003                                \newline               iter= 629, dist(L,R) previous to iteration is  1.6529216948955862E-0007                               \newline
\medbreak

\noindent
Sup distance from  fixed point p to point $T^{1258}(y_0)$ is  0.00022627982                           \newline            L1 distance from  fixed point p to point $T^{1258}(y_0)$ is  0.01670265394                            \newline            Euclidean distance from  fixed point p to point $T^{1258}(y_0)$ is  0.00141044319                     \newline            Euclidean distance from  fixed point p to point L is  0.00279389720                                   \newline            Euclidean distance from  fixed point p to point R is  0.00008060930                                   \newline            Euclidean distance between L and R is  0.00276888666           \newline            angle between $W^u_e(p)$ and iterated arc $LR$ =  3.14156 radians,         angle in degrees is $\approx$ 180                              \newline            angle between vectors $(p,L)$ and $(p,R)$ is  1.24158                      angle in degrees is $\approx$ 71                              \newline            angle between vectors $(p,T^2(L))$ and $(p,T^2(R))$ is  2.76672 radians,                                         angle in degrees is $\approx$ 159                              \newline            angle between vectors $(p,T^{4}(L))$ and $(p,T^{4} (R))$ is  3.13482 radians                                     angle in degrees is $\approx$  180                             \newline            angle between vectors $(p,T^{6}(L))$ and $(p,T^{6} (R))$ is  3.14036 radians                                      angle in degrees is $\approx$  180                             \newline            angle between vectors $(p,T^{8}(L))$ and $(p,T^{8} (R))$ is  3.13746 radians                                      angle in degrees is $\approx$  180                             \newline            angle between vectors $(p,T^{10}(L))$ and $(p,T^{10} (R))$ is  3.12554 radians                                    angle in degrees is $\approx$  179                             \newline            angle between vectors $(p,T^{12}(L))$ and $(p,T^{12} (R))$ is  3.10770 radians                                    angle in degrees is $\approx$  178                             \newline            angle between vectors $(p,T^{14}(L))$ and $(p,T^{14} (R))$ is  3.05600 radians                                    angle in degrees is $\approx$  175                             \newline
\medbreak

\noindent
Enter gap distance as a real exponent of 2 between 3 and 20                        \newline                             last exponent used is 15.03200000                              \newline               To finish the program enter exponent=0. Chosen exponent = 11.00000000                                 \newline            gap = 2048.0000000                   exponent = 11.0000000    \newline            iter= 0                              dist(L,R) previous to iteration is  3.1244649405582926E-0010    \newline               dist between L and R after applying $T^{24}$ is  7.0403463016490730E-0008                                 \newline               dist(L,L1)+dist(L1,M)+dist(M,R1)+dist(R1,R)=  7.0403463017652905E-0008                                \newline               iter= 12, dist(L,R) previous to iteration is  6.8753381889103818E-0011                                \newline               dist between L and R after applying $T^{24}$ is  1.6655785942826604E-0007                                 \newline               dist(L,L1)+dist(L1,M)+dist(M,R1)+dist(R1,R)=  1.6655785942826604E-0007                                \newline               iter= 24, dist(L,R) previous to iteration is  1.6265415954054225E-0010                                \newline               dist between L and R after applying $T^{24}$ is  4.1674332229362251E-0008                                 \newline               dist(L,L1)+dist(L1,M)+dist(M,R1)+dist(R1,R)=  4.1674332229362251E-0008                                \newline               iter= 36, dist(L,R) previous to iteration is  4.0697590197929017E-0011                                \newline               dist between L and R after applying $T^{24}$ is  3.5735658309920103E-0007                                 \newline               dist(L,L1)+dist(L1,M)+dist(M,R1)+dist(R1,R)=  3.5735658309920103E-0007                                \newline               iter= 48, dist(L,R) previous to iteration is  3.4898103807564854E-0010                                \newline               dist between L and R after applying $T^{24}$ is  3.1598931331358240E-0006                                 \newline               dist(L,L1)+dist(L1,M)+dist(M,R1)+dist(R1,R)=  3.1598931331363929E-0006                                \newline               iter= 60, dist(L,R) previous to iteration is  3.0858331376860077E-0009                                \newline               dist between L and R after applying $T^{24}$ is  8.7084375703600538E-0005                                 \newline               dist(L,L1)+dist(L1,M)+dist(M,R1)+dist(R1,R)=  8.7084375703625518E-0005                                \newline               iter= 72, dist(L,R) previous to iteration is  8.5043335648104785E-0008                                \newline               dist between L and R after applying $T^{24}$ is  1.3240893865632102E-0003                                 \newline               dist(L,L1)+dist(L1,M)+dist(M,R1)+dist(R1,R)=  1.3240893866628222E-0003                                \newline            distance between iterates is too large or curvature is big     \newline               dist between L and R after applying $T^{16}$ is  2.2719494014277907E-0004                                \newline               dist(L,L1)+dist(L1,M)+dist(M,R1)+dist(R1,R)=  2.2719494016150623E-0004                                \newline            D dist between iterates continues to be too large or curvature is big                                   \newline               dist between L and R after applying $T^{4}$ is  7.0995114872475991E-0006                                 \newline               dist(L,L1)+dist(L1,M)+dist(M,R1)+dist(R1,R)=  7.0995114872475992E-0006                                \newline               iter= 74, dist(L,R) previous to iteration is  6.9331166869485615E-0009                                \newline               dist between L and R after applying $T^{24}$ is  1.3422129968050613E-0005                                 \newline               dist(L,L1)+dist(L1,M)+dist(M,R1)+dist(R1,R)=  1.3422129968054511E-0005                                \newline
$\vdots\hspace{1cm}\vdots\hspace{1cm}\vdots\hspace{1cm}\vdots\hspace{1cm}\vdots\hspace{1cm}\vdots \hspace{1cm}\vdots\hspace{1cm}\vdots\hspace{1cm}\vdots$

\noindent
iter= 624, dist(L,R) previous to iteration is  4.5100485932166933E-0010                               \newline              dist between L and R after applying $T^{10}$ is  4.4858004094609686E-0007                                 \newline               dist(L,L1)+dist(L1,M)+dist(M,R1)+dist(R1,R)=  4.4858004094609686E-0007                                \newline               iter= 629, dist(L,R) previous to iteration is  4.3806644611849589E-0010                               \newline
\medbreak

\noindent
Sup distance from  fixed point p to point $T^{1258}(y_0)$ is  0.00022627982                           \newline            L1 distance from  fixed point p to point $T^{1258}(y_0)$ is  0.01670265394                            \newline            Euclidean distance from  fixed point p to point $T^{1258}(y_0)$ is  0.00141044319                     \newline            Euclidean distance from  fixed point p to point L is  0.00141021923                                   \newline            Euclidean distance from  fixed point p to point R is  0.00141066714                                   \newline            Euclidean distance between L and R is  0.00000044858           \newline            angle between $W^u_e(p)$ and iterated arc $LR$ =  0.00003 radians,           angle in degrees is $\approx$ 0                                \newline            angle between vectors $(p,L)$ and $(p,R)$ is  0.00002                      angle in degrees is $\approx$ 0                               \newline            angle between vectors $(p,T^2(L))$ and $(p,T^2(R))$ is  0.00000 radians,                                            angle in degrees is $\approx$ 0                                \newline            angle between vectors $(p,T^{4}(L))$ and $(p,T^{4} (R))$ is  0.00000 radians                                      angle in degrees is $\approx$  0                               \newline            angle between vectors $(p,T^{6}(L))$ and $(p,T^{6} (R))$ is  0.00000 radians                                      angle in degrees is $\approx$  0                               \newline            angle between vectors $(p,T^{8}(L))$ and $(p,T^{8} (R))$ is  0.00000 radians                                      angle in degrees is $\approx$  0                               \newline            angle between vectors $(p,T^{10}(L))$ and $(p,T^{10} (R))$ is  0.00000 radians                                    angle in degrees is $\approx$  0                               \newline            angle between vectors $(p,T^{12}(L))$ and $(p,T^{12} (R))$ is  0.00001 radians                                   angle in degrees is $\approx$  0                               \newline            angle between vectors $(p,T^{14}(L))$ and $(p,T^{14} (R))$ is  0.00009 radians                                    angle in degrees is $\approx$  0                               \newline            Enter gap distance as a real exponent of 2 between 3 and 20    \newline            last exponent used is 11.00000000                              \newline               To finish the program enter exponent=0. Chosen exponent =  0.00000000                                 \newline            Press ENTER  to finish the program.                      \newline

  }

\bigbreak

J.~J.~Nieto, Facultad de Matem\'{a}ticas, Universidad de Santiago de Compostela,

Santiago de Compostela, La Coruña, España.

juanjose.nieto.roig@usc.es

\medbreak

M.~J.~Pacifico, Instituto de Matematica, Universidade Federal do Rio de Janeiro,

C. P. 68.530, CEP 21.945-970, Rio de Janeiro, R. J. , Brazil

pacifico@im.ufrj.br

\medbreak

J.~L.~Vieitez, Regional Norte, Universidad de la Republica,

Rivera 1350, CP 50000, Salto, Uruguay

 jvieitez226@gmail.com


\begin{thebibliography}{Wu-Du, 1979}
\bibitem [Ims2]{Ims2}\cit{R. A. Ims}
{`` It is a reasonable assumption that sex ratio is 1:1.
We know that this for sure at birth (lab studies). ''}
{Personal communication}{}
\end{thebibliography}
\end{document}